\documentclass[reqno,12pt]{amsart}
\usepackage[margin=1in]{geometry}
\usepackage{amsfonts}
\usepackage[all]{xy}
\usepackage{enumitem}
\usepackage{mathtools}
\usepackage{amssymb,amsmath,url}
\usepackage{setspace}
\usepackage{amsthm}
\usepackage{cite}
\usepackage{color}
\numberwithin{equation}{section}
\usepackage[normalem]{ulem}
\usepackage{graphicx}
\usepackage{lmodern}
\graphicspath{ {images/} }
\usepackage{pdfpages}

\author{Zhenghui Huo and Brett D. Wick}
\title{Weighted estimates for the Bergman projection on the Hartogs triangle}
\begin{document}
	\thanks{BDW's research is partially supported by National Science Foundation grants DMS \# 1560955 and DMS \# 1800057 and Australian Research Council grant DP 190100970}.
	\address{Zhenghui Huo, Department of Mathematics and Statistics, The University of Toledo,  Toledo, OH 43606-3390, USA}
\email{zhenghui.huo@utoledo.edu}
\address{Brett D. Wick, Department of Mathematics and Statistics, Washington University in St. Louis,  St. Louis, MO 63130-4899, USA}
\email{wick@math.wustl.edu}
		\newtheorem{thm}{Theorem}[section]
	\newtheorem{cl}[thm]{Claim}
	\newtheorem{lem}[thm]{Lemma}
		\newtheorem*{Rmk*}{Remark}
	\newtheorem{ex}[thm]{Example}
	\newtheorem{de}[thm]{Definition}
	\newtheorem{co}[thm]{Corollary}
	\newtheorem*{thm*}{Theorem}
		\theoremstyle{definition}
		\newtheorem{rmk}[thm]{Remark}
	\maketitle
\begin{abstract}
We apply modern techniques of dyadic harmonic analysis to obtain sharp estimates for the Bergman projection in weighted Bergman spaces. Our main theorem focuses on the Bergman projection on the Hartogs triangle. The estimates of the operator norm are in terms of a Bekoll\'e-Bonami type constant.  As an application of the results obtained, we give, for example, an upper bound for the $L^p$ norm of the Bergman projection on the generalized Hartogs triangle $\mathbb H_{m/n}$ in $\mathbb C^2$.


\medskip

\noindent
{\bf AMS Classification Numbers}: 32A25, 32A36,  32A50, 42B20, 42B35

\medskip

\end{abstract}

\section{Introduction}
Let $\Omega\subseteq \mathbb C^n$ be a bounded domain. Let $L^2(\Omega)$ denote the space of square-integrable functions with respect to the Lebesgue measure $dV$ on $\Omega$. Let $A^2(\Omega)$ denote the subspace of square-integrable holomorphic functions. The Bergman projection $P$ is the orthogonal projection from $L^2(\Omega)$ onto $A^2(\Omega)$. Associated with $P$, there is a unique function $K_\Omega$ on $\Omega\times\Omega$ such that for any $f\in L^2(\Omega)$:
\begin{equation}
P(f)(z)=\int_{\Omega}K_\Omega(z;\bar w)f(w)dV(w).
\end{equation}
Let $P^+$ denote the positive Bergman projection defined by:
\begin{equation}
P^+(f)(z):=\int_{\Omega}|K_\Omega(z;\bar w)|f(w)dV(w).
\end{equation}
A question of importance in analytic function theory and harmonic analysis is to understand the boundedness of $P$ or $P^+$ on the space $L^p(\Omega, \mu dV)$, where $\mu$ is some non-negative locally integrable function on $\Omega$.

For the unweighted case ($\mu\equiv 1$), the $L^p$ boundedness for the Bergman projection have been studied in various settings. On a wide class of domains,  the Bergman projection is $L^p$ regularity for all $1<p<\infty$. See for instance \cite{Fefferman,PS,McNeal1,McNeal3,NRSW,McNeal3,McNeal2,MS,CD,EL,BS}. In all these results, the domain needs to satisfy certain boundary conditions. On some other domains, the projection has only a finite range of mapping regularity. See for example \cite{Yunus,DebrajY,EM,EM2,CHEN}. One important example is the Hartogs triangle $\mathbb H$. In \cite{DebrajY}, Chakrabarti and Zeytuncu showed that the Bergman projection on the Hartogs triangle is $L^p$-regular if and only if $\frac{4}{3}<p<4$. 

Less is known about the situation when the weight $\mu\not\equiv 1$, and results and progress depend upon the domains being studied.  For the case of the unit ball in $\mathbb{C}^n$, the boundedness of $P$ and $P^+$ in the weighted $L^p$ space was studied by Bekoll\'e and Bonami in \cite{BB78} and \cite{Bekolle}.  Let $T_z$ denote the Carleson tent over $z$ in the unit ball $\mathbb B_n$ defined as below:
 \begin{itemize}
 	\item $T_z:=\left\{w\in \mathbb B_n:\left|1-\bar w\frac{z}{|z|}\right|<1-|z|\right\}$ for $z\neq 0$, and
 	\item $T_z:= \mathbb B_n$ for $z=0$.
 \end{itemize} Then the result of Bekoll\'e and Bonami can be stated as follows:
 \begin{thm}(Bekoll\'e-Bonami)
 Let the weight $\mu(w)$ be a positive, locally integrable function on the unit ball $\mathbb B_n$. Let $1<p<\infty$. Then the following conditions are equivalent:
 \begin{enumerate}
 	\item $P:L^p(\mathbb B_n,\mu)\mapsto L^p(\mathbb B_n,\mu)$ is bounded.
 		\item $P^+:L^p(\mathbb B_n,\mu)\mapsto L^p(\mathbb B_n,\mu)$ is bounded.
 		\item The Bekoll\'e-Bonami constant $$B_p(\mu):=\sup_{z\in \mathbb B_n}\frac{\int_{T_z}\mu(w) dV(w)}{\int_{T_z}dV(w)}\left(\frac{\int_{T_z}\mu^{-\frac{1}{p-1}} (w)dV(w)}{\int_{T_z}dV(w)}\right)^{p-1}$$ is finite.
 \end{enumerate} 
 \end{thm}
Motivated by recent developments on the $A_2$-Conjecture \cite{Hytonen} for singular integrals in the setting of Muckenhoupt weighted $L^p$ spaces, people have made progress on the dependence of the operator norm $\|P\|_{L^p(\mathbb B_n,\mu)}$ on $B_{p}(\mu)$. In \cite{Pott}, Pott and Reguera gave a weighted $L^p$ estimate for the Bergman projection on the upper half plane. Their estimates are in terms of the Bekoll\'e-Bonami constant and the upper bound estimate is sharp. Later, Rahm, Tchoundja, and Wick \cite{Rahm} generalized the results of Pott and Reguera to the unit ball case and also obtained estimates for the Berezin transform.

The purpose of this paper is to establish sharp weighted inequalities for the Bergman projection on the Hartogs triangle $\mathbb H$. {The Hartogs triangle is a bounded pseudoconvex domain defined by $\mathbb H=\{(z_1,z_2)\in\mathbb C^2:|z_1|<|z_2|<1\}$. The boundary $\mathbf b\mathbb H$ of $\mathbb H$ has a serious singularity at the origin, where $\mathbf b\mathbb H$ cannot be represented as a graph of a continuous function. Partially because of this, $\mathbb H$ exhibits many interesting phenomena unseen on smooth domains and serves as a source of counterexamples to many conjectures in  several complex variables. The closure $\bar {\mathbb H}$ does not have a Stein neighborhood basis. The $\bar \partial$ problem on $\mathbb H$ is not global regular \cite{CC}, i.e. there exists a $\bar \partial$-closed $(0,1)$ form $h\in C^\infty_{0,1} (\bar {\mathbb H})$ such that no solution $u$ of the equation $\bar \partial u=h$ is in $C^\infty(\bar {\mathbb H})$. The Bergman projection on $\mathbb H$ has only limited $L^p$ regularity for $p\in (3/4,4)$ \cite{DebrajY}. This makes even the unweighted  $L^p$ norm estimate of the projection interesting.}

We give a Bekoll\'e-Bonami type constant and obtain weighted $L^p$-norm estimates for $P$ and $P^+$.  Recall that the Hartogs triangle $\mathbb H$ is defined by
\begin{equation*}
	\mathbb H=\{(z_1,z_2)\in \mathbb C^2:|z_1|<|z_2|<1\}.
\end{equation*}
{$\mathbb H$ is biholomorphic to the product domain of the disc and the punctured disc. By the biholomorphic transformation formula, the kernel $K_{\mathbb H}(z_1,z_2; \bar w_1, \bar w_2)$ has the following form:
	\begin{align*}
K_{\mathbb H}(z_1,z_2;\bar w_1,\bar w_2)&=\frac{1}{\pi^2z_2\bar w_2(1-\frac{z_1\bar w_1}{z_2 \bar w_2})^2(1-z_2\bar w_2)^2}.
	\end{align*}
Detailed computation for $K_{\mathbb H}$ is provided in the next section.}
	Given functions of several variables $f$ and $g$, we use $f\lesssim g$ to denote that $f\leq Cg$ for a constant $C$. If $f\lesssim g$ and $g\lesssim f$, then we say $f$ is comparable to $g$ and write $f\approx g$. {For a weight $\mu$ and a subset $U$ in a domain $\Omega$, we set $\mu(U):=\int_U\mu dV$ and let $\langle f\rangle^{\mu dV}_U$ denote the average of the function $|f|$ with respect to the measure $\mu dV$ on the set $U$:
	\begin{equation}\label{2.1100}
	\langle f\rangle^{\mu dV}_U=\frac{\int_{U}|f(w_1,w_2)|\mu dV}{\mu(U)}.
	\end{equation}}
	
The main result obtained in this paper is:
\begin{thm}
\label{t:main}
	Let $1<p<\infty$, and $p'$ denote the H\"older conjugate to $p$.
	Let $\mu$ be a positive, locally integrable weight on $\mathbb H$  of the form
	\begin{align}
	\label{1.4}
\mu(z_1,z_2)=\mu_{1}(z_1/z_2)\mu_{2}(z_2).\end{align}
  Set $\nu=|z_2|^{-p^\prime}\mu^{\frac{-p^\prime}{p}}$  and $du=|z_2|^{-2}dV$. Then the Bergman projection $P$ is bounded on the weighted function space $L^p(\mathbb H,\mu dV)$ if and only if $[\mu,\nu]_p<\infty$. 
  
Moreover, the following  quantitative estimate is provided:
	\begin{align}\label{2.102}
	[\mu,\nu]_p^{\frac{1}{2p}}\lesssim\|P\|_{L^p(\mathbb H,\mu dV)}\leq \|P^+\|_{L^p(\mathbb H,\mu dV)}\lesssim &([\mu,\nu]^{0,0}_p)^{\frac{1}{p}}+pp^\prime([\mu,\nu]^{1,0}_p+[\mu,\nu]^{0,1}_p)+\nonumber\\&(pp^\prime)^2([\mu,\nu]^{1,1}_p)^{\max\{1,\frac{1}{p-1}\}}.
	\end{align}
Here 
		\begin{align}\label{2.103}
&	[\mu,\nu]_p=\sup_{z_1,z_2\in \mathbb D}\langle\mu|w_2|^{2-p} \rangle^{du}_{T^\prime_{z_1,z_2}}\left(\langle |w_2|^{2}\nu\rangle^{du}_{T^\prime_{z_1,z_2}}\right)^{p-1};\\&	[\mu,\nu]^{0,0}_p=\langle |w_2|^{2-p}\mu\rangle^{du}_{\mathbb H}\left(\langle\nu|w_2|^2 \rangle^{du}_{\mathbb H}\right)^{p-1};
\\&	[\mu,\nu]^{1,0}_p=\left(\langle|w_2|^{2-p}\mu_2\rangle^{dV}_{\mathbb D}(\langle|w_2|^2\nu_2\rangle^{dV}_{\mathbb D})^{p-1}\right)^{\frac{1}{p}}\left(\sup_{\substack{z\in \mathbb D,\\|z|>{1}/{2}}}\langle\mu_1\rangle^{dV}_{T_z}\left(\langle\nu_1\rangle^{dV}_{T_z}\right)^{p-1}\right)^{\max\{1,\frac{1}{p-1}\}};\\&	[\mu,\nu]^{0,1}_p=\left(\sup_{\substack{z\in \mathbb D,\\|z|>{1}/{2}}}\langle|w_2|^{2-p}\mu_2\rangle^{dV}_{T_z}(\langle|w_2|^2\nu_2\rangle^{dV}_{T_z})^{p-1}\right)^{\max\{1,\frac{1}{p-1}\}}\left(\langle\mu_1\rangle^{dV}_{\mathbb D}\left(\langle\nu_1\rangle^{dV}_{\mathbb D}\right)^{p-1}\right)^{\frac{1}{p}};\\&	[\mu,\nu]^{1,1}_p=
\sup_{\substack{z_1,z_2\in \mathbb D\\|z_1|>1/2,|z_2|>1/2}}\langle\mu|w_2|^{2-p} \rangle^{du}_{T^\prime_{z_1,z_2}}\left(\langle |w_2|^{2}\nu\rangle^{du}_{T^\prime_{z_1,z_2}}\right)^{p-1}.\end{align}
\end{thm}
{ For the definitions of the induced Carleson tents $T^\prime_{z_1,z_2}$ of the Hartogs triangle, see Section 2.\\

\begin{rmk}
	The constant $[\mu,\nu]_p$ serves as a natural generalization of the $B_p$ constant for the Hartogs triangle case.
	It is not hard to see that $[\mu,\nu]_p$ and the upper bound in Theorem \ref{t:main} are qualitatively equivalent, i.e. $[\mu,\nu]_p$ is finite if and only if the sum \[([\mu,\nu]^{0,0}_p)^{\frac{1}{p}}+pp^\prime([\mu,\nu]^{1,0}_p+[\mu,\nu]^{0,1}_p)+(pp^\prime)^2([\mu,\nu]^{1,1}_p)^{\max\{1,\frac{1}{p-1}\}}\] is finite. But they are not quantitatively equivalent. More specifically, $[\mu,\nu]_p$ and the upper bound satisfy the following inequalities:
	\begin{align}\label{1.11}
([\mu,\nu]_p)^{\frac{1}{p}}\lesssim&([\mu,\nu]^{0,0}_p)^{\frac{1}{p}}+pp^\prime([\mu,\nu]^{1,0}_p+[\mu,\nu]^{0,1}_p)+(pp^\prime)^2([\mu,\nu]^{1,1}_p)^{\max\{1,\frac{1}{p-1}\}}\nonumber\\\lesssim& (pp^\prime)^2([\mu,\nu]_p)^{\max\{1,\frac{1}{p-1}\}}
	\end{align}
	 As one will see in the proof of Theorem \ref{t:main}, the products of averages of $\mu$ and $\nu$ over different tents will have different impacts on the estimate for the weighted norm of the projection $P$. The constant $[\mu,\nu]_p$ above fails to reflect such a difference, and hence is unable to give the sharp upper bound. This issue did not occur in the upper half plane case \cite{Pott} since the average over the whole upper half plane is not included in the $B_p$ constant there. 
\end{rmk}

\begin{rmk}
	In Theorem \ref{t:main}, we consider the weight $\mu$ of the form as in (\ref{1.4}) so that the boundedness of the weighted maximal operator in Lemma \ref{l:2.7} follows by the Fubini's theorem.  See Section 5 for further discussion on this assumption. The measure $du$ on $\mathbb H$ is induced by the Lebesgue measure on $\mathbb D^2$. The weight $\nu$ is chosen to be the dual weight of $|z_2|^{-p}\mu$ with respect to the measure $du$ so that a similar argument as in \cite{Pott} and \cite{Rahm} works for the Hartogs triangle case.
\end{rmk}}

There has been some recent interest in analyzing the $L^p$ regularity properties of the projection via characteristics of the weight.  In \cite{Liwei17}, Chen considered an $A^+_p$ condition, which is equivalent to the Bekoll\'e-Bonami condition in the upper half plane setting, and obtained the $L^p$ regularity of the weighted Bergman projection with some special weights on the Hartogs triangle. Using the $A_p^+$ condition, Chen, Krantz, and Yuan \cite{Liwei19} obtained the $L^p$ regularity results for the Bergman projections on domains covered by the polydisc through a rational proper holomorphic map.  The result of Chakrabarti and Zeytuncu in  \cite{DebrajY} can be recovered from \cite{Liwei17} by showing that the $A^+_p$ constant of the weight $\mu\equiv 1$ blows up for $p\notin(\frac{4}{3},4)$. Similarly, Theorem \ref{t:main} provide another proof for this result.

The approach we employ in this paper is similar to the ones in \cite{Pott} and \cite{Rahm}. The lower bound estimate follows from a weak-type inequality argument. To obtain the upper bound estimate, we show that $P$ and $P^+$ are controlled by a positive dyadic operator.  Then an analysis on the weighted $L^p$ norm of the dyadic operator yields the desired estimate. Here we use harmonic analysis strategy from \cite{Moen2012} and \cite{Lacey2017}. In particular, we build the dyadic structure on the Hartogs triangle induced by the dyadic structure on the unit disc via the biholomorphism between $\mathbb H$ and $\mathbb D\times \mathbb D^*$.  We also use techniques from  multi-parameter harmonic analysis to control the induced product structure on the Hartogs triangle. See also Remark 5 in Section 5.  Our upper bound is sharp. In Section 4.1, we provide an example of weights and functions where the sharp bound is attained. As applications of our results, we recover the $L^p$-regularity results in \cite{DebrajY} and \cite{EM2}  and give upper bound estimates for the $L^p$-norm of the Bergman projections on the Hartogs triangle $\mathbb H$ and the generalized Hartogs triangle $\mathbb H_{m/n}$.   See Sections 4.2 and 4.4. It is worth noting that the construction of the positive dyadic operator relies on a dyadic structure on the unit disc where the measure of the set in the structure can be used to estimate the Bergman kernel function. Since the dyadic structures on the disc $\mathbb D$ and the ball $\mathbb B_n$ are well understood, the approach we use in this paper can also be applied to the setting where the domain is related to the unit disc or ball, such as the polydisc, the product of unit balls, and domains that are biholomorphically equivalent to them.

The paper is organized as follows: In Section 2, we introduce a dyadic structure on the unit disc and a corresponding structure on the Hartogs triangle and provide the results that will be used throughout the paper. In Section 3, we present the dyadic operator $Q^+_{m,n,\nu}$ and prove Theorem \ref{t:main}. In Section 4, we give a sharp example for our upper bound estimate. We also provide some examples where the upper bound estimates can be explicitly computed. In Section 5, we make several remarks and possible directions for generalization.
\vskip 10pt
\paragraph{\bf Acknowledgments}
B. D. Wick's research is partially supported by National Science Foundation grants DMS \# 1560955 and DMS \# 1800057 and Australian Research Council grant DP 190100970. 
We would like acknowledge Liwei Chen, John D'Angelo, and {the referee} for their suggestions and comments. 
\section{Preliminaries}
 Let $\mathbb D$ denote the unit disc in $\mathbb C$. Let $\mathbb D^*$ denote the punctured disc $\mathbb D\backslash \{0\}$. 
The Hartogs triangle $\mathbb H$ is defined by
\begin{equation}
	\mathbb H=\{(z_1,z_2)\in \mathbb C^2:|z_1|<|z_2|<1\}.
\end{equation}
Note that the mapping $(z_1,z_2)\mapsto (\frac{z_1}{z_2},z_2)$ is a biholomorphism from $\mathbb H$ onto $\mathbb D\times \mathbb D^*$. The biholomorphic transformation formula (see \cite{Krantz}) then implies that 
	\begin{align}\label{K}
K_{\mathbb H}(z_1,z_2;\bar w_1,\bar w_2)&=\frac{1}{z_2\bar w_2}K_{\mathbb D\times\mathbb D^*}\left(\frac{z_1}{z_2},z_2;\frac{\bar w_1}{\bar w_2},\bar w_2\right)\nonumber\\
&=\frac{1}{z_2\bar w_2}K_{\mathbb D\times\mathbb D}\left(\frac{z_1}{z_2},z_2;\frac{\bar w_1}{\bar w_2},\bar w_2\right)\nonumber\\
&=\frac{1}{\pi^2z_2\bar w_2(1-\frac{z_1\bar w_1}{z_2 \bar w_2})^2(1-z_2\bar w_2)^2}.
	\end{align}	
Hence, the Bergman projection $P$ and the absolute Bergman projection $P^+$ on the Hartogs triangle can be expressed as follows
\begin{align}
P(f)(z)&=\int_{\mathbb H}\frac{f(w)}{\pi^2z_2\bar w_2(1-\frac{z_1\bar w_1}{z_2 \bar w_2})^2(1-z_2\bar w_2)^2}dV(w);\\P^+(f)(z)&=\int_{\mathbb H}\frac{f(w)}{\pi^2|z_2w_2||1-\frac{z_1\bar w_1}{z_2 \bar w_2}|^2|1-z_2\bar w_2|^2}dV(w).
\end{align}
	
	We next introduce a dyadic structure on the unit disk. A related construction appears in \cite{Rochberg}. Let $\mathcal D=\{D^k_j\}$ be a dyadic system on the unit circle with
	$$D^k_j=\{e^{2\pi i\theta}:(j-1)2^{-k}\leq \theta<j2^{-k}\}, \;\text{ for } \;j=1,\dots, 2^k. $$
	Let $d(\cdot,\cdot)$ denote the Bergman metric on the unit disc $\mathbb D$. For $z\in \mathbb D$, let $B(z,r)$ denote the ball centered at point $z$ with radius $r$ under this metric. Set $r=2^{-1}\ln 2$. For $k\in \mathbb N$, let $\mathbb S_{kr}$ denote the circle centered at the origin with radius $kr$ in the Bergman metric.  Let $\mathcal P_{kr}z$ be the radial projection of $z$ onto the sphere $\mathbb S_{Nr}$. By the proof of \cite[Lemma 9]{Rahm}, $\{\mathcal P_{k\theta}D^k_j\}$ satisfy the following three properties:
	\begin{enumerate}
		\item $\mathbb S_{kr}=\cup^{2^k}_{j=1}\mathcal P_{kr}D^k_j$;
		\item $\mathcal P_{kr}D^k_j\cap \mathcal P_{kr}D^k_i=\emptyset$ for $i\neq j$;
		\item For $w^k_j=\mathcal P_{kr}e^{2\pi i(j-\frac{1}{2})2^{-k}}$, $\mathbb S_{kr}\cap B(w^k_j,\lambda)\subseteq\mathcal  P_{kr}D^k_j\subseteq \mathbb S_{kr}\cap B(w^k_j,C\lambda)$.
	\end{enumerate}
Define subsets, $K_j^k$ of $\mathbb D$ to be: 
\begin{align*}
K_1^0&:=\{z\in \mathbb D:d(0,z)<r\};\\
K^k_j&:=\{z\in \mathbb D:kr\leq d(0,z)<(k+1)r\; \text{ and } \; \mathcal P_{kr}z\in\mathcal  P_{kr} D_j^k \}, k\geq 1, j\geq 1.
\end{align*}
For $k=0$ and $j=1$, set $c^0_1\in K^0_1$ to be the origin. For $k\geq 1$, set $c_j^k\in K_j^k$ to be the point $\mathcal P_{(k+\frac{1}{2})r}w_j^k$. For $\alpha=c_j^k$, the set $K_\alpha:=K_j^k$ is referred to as a kube and the point $\alpha=c_j^k$ is the center of the kube. We define a Bergman tree structure $\mathcal T:=\{c_j^k\}$ on centers of the kubes. We say that $c_i^{k+1}$ is a child of $c^k_j$ if $\mathcal P_{kr}D_i^{k+1}\subseteq\mathcal  P_{kr}D_j^k$. We say $c_i^m\geq c_j^k$ if {$m\geq k$} and $\mathcal P_{kr}c_i^m\in \mathcal P_{kr} D_j^k$. We define $\hat K_\alpha$ to be the dyadic tent under $K_\alpha$:
\begin{align}
\hat K_\alpha:=\bigcup_{\beta\in\mathcal T:\beta\geq \alpha}K_\beta.
\end{align}
For $z\in\mathbb D$, we say the generation $\text{gen}(z)=N$ if $z\in K^N_j$ for some $j$. 

Using shifted dyadic systems $\mathcal D_l=\{D^k_j(l)\}$ on the unit circle  with $$D^k_j(l)=\{e^{2\pi i\theta}:(j-1)2^{-k}+l\leq \theta<j2^{-k}+l\}, \;\text{ for } \;j=1,\dots, 2^k \text{ and } l\in \mathbb R,$$
one can obtain different dyadic structures on $\mathbb D$ with their corresponding Bergman trees $\mathcal T_l$.
Recall the Carleson tent $T_z$ over $z\in  \mathbb D$:
\begin{itemize}
	\item $T_z:=\left\{w\in \mathbb D:\left|1-\bar w\frac{z}{|z|}\right|<1-|z|\right\}$ for $z\neq 0$, and
	\item $T_z:= \mathbb D$ for $z=0$.
\end{itemize} 
For a subset $U$, we use the notation $|U|$ to denote the Lebesgue measure of $U$. The following three lemmas relate the Carleson tent $T_z$ to the dyadic tent $\hat K_\alpha$ and the Bergman kernel function on $\mathbb D$. 
\begin{lem}
	\label{l:2.1}
	Let $\mathcal T$ be a Bergman tree constructed as above. For $\alpha\in \mathcal T$,
	$$|T_\alpha|\approx|\hat K_\alpha|\approx|K_\alpha|\approx(1-|\alpha|)^{2}.$$
\end{lem}
\begin{proof}Suppose $\text{gen}(\alpha)=k$. Let $R_{kr}$ denote the Euclidean distance between $\mathbb S_{kr}$ and the origin. Then $|\hat K_\alpha|=\pi 2^{-k} (1-R_{kr}^2)$ and $|K_\alpha|=\pi 2^{-k} (R_{(k+1)r}^2-R_{kr}^2)$. Recall that $r=2^{-1}\ln 2$. By the definition of the Bergman distance, $1-R_{kr}\approx e^{-2kr}=2^{-k}$. Thus $|\hat K_\alpha|\approx|K_\alpha|\approx 2^{-2k}$. Since $\alpha$ is the center of the kube $K_\alpha$, the Bergman distance $d(0,\alpha)=(k+\frac{1}{2})r$. Hence we obtain $$(1-|\alpha|)^2=(1-R_{(k+\frac{1}{2})r})^2\approx 2^{-2(k+\frac{1}{2})}\approx|\hat K_\alpha|\approx|K_\alpha|.$$ 
	
	Notice that the Carleson tent $T_\alpha$ is the intersection set of the unit disc $\mathbb D$ and the disc centered at the point $\frac{z}{|z|}$ with Euclidean radius $1-|\alpha|$. A geometric consideration then yields $$|T_\alpha|\approx(1-|\alpha|)^2.$$
\end{proof}
\begin{lem}[\protect{\cite[Lemma 9]{Rahm}}]
\label{l:2.2}
	There is a finite collection of Bergman trees $\{\mathcal T_l\}^N_{l=1}$ such that for all $\alpha\in \mathbb D$, there is a tree $\mathcal T$ from the finite collection and an $\beta\in \mathcal T$ such that the dyadic tent $\hat K_\beta$ contains the tent $T_\alpha$ and $\sigma (\hat K_\beta)\approx|T_\alpha|$.
\end{lem}

\begin{lem}[{\cite[Lemma 15]{Rahm}}]
\label{l:2.3}
	For $z,w\in \mathbb D$, there is a Carleson tent, $T_\alpha$, containing $z$ and $w$ such that \begin{equation}\label{2.4}|T_\alpha|\approx|1-z\bar w|^2=\pi^{-1}|K_{\mathbb D}(z,\bar w)|^{-1}. \end{equation}
\end{lem}

\begin{lem}
\label{l:2.4}
For any dyadic tent $\hat K_\beta$ with $\beta\in \mathcal T_l$ for some $l$, there exists a Carleson tent $T_z$ such that $\hat K_\beta\subseteq T_z$ and $|\hat K_\beta|\approx|T_z|$.
\end{lem}
\begin{proof} Given a dyadic tent $\hat K_\beta$, we can find a Carleson tent $T_z$ such that $\hat K_\beta$ is a largest dyadic tent in $T_z$. Without loss of generality, we may assume that $z$ is a positive real number.
By Lemma \ref{l:2.1}, $|\hat K_\alpha|\approx|K_\alpha|$. It suffices to show that the top kube $K_\beta$ of the tent $\hat K_\beta$ satisfies the inequality $|K_\beta|\approx|T_z|$. Since $K_\beta$ is a largest kube contained in $T_{z}$, all of its ancestors are not contained in $T_z$. Let $k$ be the generation $\text{gen}(\beta)$ of $\beta$. Then  $T_z$ intersects with at most two of the Borel subsets $\{Q^{k-1}_j\}^{2^{k-1}}_{j=1}$ of $S_{(k-1)\theta}$. Let $R_{(k-1)r}$ denote the Euclidean distance between $\mathbb S_{(k-1)r}$ and the origin. The arc length of the set $P_{(k-1)r}D^{k-1}_j$ equals $R_{(k-1)r}2\pi2^{1-k}$. Thus the arc length of  the intersection set $\mathbb S_{(k-1)r}\cap T_t$ is less than $2R_{(k-1)r}2\pi2^{1-k}$. Note that the point $z$ is a positive real number. $T_z$ is symmetric about the real number axis. Therefore the point $R_{(k-1)r}e^{2\pi i2^{1-k}}$ is not in $T_z$, i.e.
$$|1-R_{(k-1)r}e^{2\pi i2^{1-k}}|\geq 1-z.$$
Since $1-R_{Nt} \approx e^{-2Nt}$ and $|1-e^{2\pi it}|\approx t$ for $t\in \mathbb R$, we have \begin{align*}|1-R_{(k-1)r}e^{2\pi i2^{1-k}}|&\leq|1-R_{(k-1)r}|+ |R_{(k-1)r}-r_{(k-1)\theta}e^{2\pi i2^{1-k}}|\\&\approx e^{-2(k-1)r}(1+2^{1-k})=e^{-(k-1)\ln 2}(1+2^{1-k})\approx 2^{-(k-1)}.\end{align*}
Hence $2^{-(k-1)}\gtrsim 1-z=1-|z|$. Lemma \ref{l:2.1} then implies that $|T_z|\lesssim 2^{-2(k-1)}$. Since $\text{gen}(\beta)=k$, the Bergman distance $d(\beta,0)$ equals $(k+\frac{1}{2})r$. Recall that $r=2^{-1}\ln 2 $. We have $$1-|\beta|\approx e^{-2(k+\frac{1}{2})\theta}=2^{-(k+\frac{1}{2})}.$$ Applying Lemma \ref{l:2.1} again yields $|K_\beta|\approx 2^{-2(k+\frac{1}{2})}\gtrsim |T_z|$. By the containment $K_\beta\subseteq T_z$, there holds $|K_\beta|\leq |T_z|$. Combining these inequalities, we conclude that  $|K_\beta|\approx |T_z|$ and the proof is complete.
\end{proof}

Combining Lemmas \ref{l:2.2} and \ref{l:2.3}, we obtain the following estimate for arbitrary $z,w\in \mathbb D$:
\begin{align}\label{2.5}
|1-z\bar w|^{-2}\approx|T_\alpha|^{-1}\approx |\hat K_\beta|^{-1}\leq\sum_{m=1}^{M}\sum_{\gamma\in \mathcal T_m}\frac{1_{\hat K_\gamma}(z)1_{\hat K_\gamma}(w)}{|\hat K_\gamma|}.
\end{align}
Here $\{\mathcal T_m\}_{m=1}^M$ is the finite collection in Lemma \ref{l:2.2}.

Similarly, on the bidisk,  $\mathbb D^2$, we have:
\begin{align}\label{2.6}
&|1-z_1\bar w_1|^{-2}|1-z_2\bar w_2|^{-2}\nonumber\\\approx&|T_{\alpha_1}|^{-1}|T_{\alpha_2}|^{-1}\nonumber\\\approx& |\hat K_{\beta_1}|^{-1}|\hat K_{\beta_2}|^{-1}\nonumber\\\leq&\sum_{m,n=1}^{M}\sum_{\gamma\in \mathcal T_m,\eta\in \mathcal T_n}\frac{1_{\hat K_\gamma\times\hat K_\eta}(z_1,z_2)1_{\hat K_\gamma\times\hat K_\eta}(w_1,w_2)}{|\hat K_\gamma\times\hat K_\eta|}.
\end{align}
Given a tree structure $\mathcal T_m\times \mathcal T_n$ on $\mathbb D^2$ and a dyadic tent $\hat K_{\beta_1}\times \hat K_{\beta_2}$ we define the induced tree structure $\mathcal T^\prime_{m,n}$ and dyadic tent $\hat K^\prime_{\beta_1,\beta_2}$ on $\mathbb H$ to be:
\begin{align}\label{2.7}
\mathcal T^\prime_{m,n}&:=\left\{(c_1,c_2)\in \mathbb H: \left(\frac{c_1}{c_2},c_2\right)\in \mathcal T_m\times \mathcal T_n\right\},\\\label{2.8}
\hat K^\prime_{\beta_1,\beta_2}&:=\left\{(z_1,z_2)\in \mathbb H:\left(\frac{z_1}{z_2},z_2\right)\in \hat K_{\beta_1}\times \hat K_{\beta_2}\right\}.
\end{align}
Similarly the induced Carleson tent $T^\prime_{z_1,z_2}$ on $\mathbb H$ can be defined by
\begin{equation}
T^\prime_{z_1,z_2}:=\{(w_1,w_2)\in\mathbb H:\left(\frac{w_1}{w_2},w_2\right)\in T_{z_1}\times T_{z_2}\}.
\end{equation}
Set $du=|w_2|^{-2}dV$. For a weight $\mu$ and a subset $U\subseteq \mathbb H$, we set $\mu(U):=\int_U\mu dV$ and let $\langle f\rangle^{\mu dV}_U$ denote the average of the function $|f|$ with respect to the measure $\mu dV$ on the set $U$:
\begin{equation}\label{2.110}
\langle f\rangle^{\mu dV}_U=\frac{\int_{U}|f(w_1,w_2)|\mu dV}{\mu(U)}.
\end{equation}
Given weights $\mu$ on $\mathbb H$ and $\nu=|z_2|^{-p^\prime}\mu^{-p^\prime/p}$, we define the characteristic of two weights $\mu, \nu$ to be
\begin{equation}\label{2.111}
[\mu,\nu]_p:=\sup_{z_1,z_2\in \mathbb D}\langle\mu|w_2|^{2-p} \rangle^{du}_{T^\prime_{z_1,z_2}}\left(\langle |w_2|^{2}\nu\rangle^{du}_{T^\prime_{z_1,z_2}}\right)^{p-1}.
\end{equation} 
By Lemmas \ref{l:2.2} and \ref{l:2.4}, we can replace $T^\prime_{z_1,z_2}$ by $\hat K^\prime_{\gamma,\eta}$ to obtain a quantity of comparable size:
\begin{equation}
[\mu,\nu]_p\approx\sup_{1\leq m,n\leq M}\sup_{(\gamma,\eta)\in \mathcal T^\prime_{m,n}}\langle\mu|w_2|^{2-p} \rangle^{du}_{\hat K^\prime_{\gamma,\eta}}\left(\langle |w_2|^{2}\nu\rangle^{du}_{\hat K^\prime_{\gamma,\eta}}\right)^{p-1}.
\end{equation} 
From now on, we will abuse the notations $[\mu,\nu]_p$ and $[\mu,\nu]^{i,j}_p$ for $i,j=0,1$ to represent both the supremum in $T^\prime_{z_1,z_2}$ and the supremum in the corresponding $\hat K^\prime_{\gamma,\eta}$ of similar size.

The proof of Theorem \ref{t:main} will use the weighted strong maximal function on $\mathbb H$.
\begin{de}For a weight $\mu$, and a Bergman tree $\mathcal T^\prime_{m,n}$, we define the following maximal function:
\begin{equation}
\mathcal M_{\mathcal T^\prime_{m,n},\mu}f(w_1,w_2):=\sup_{(\beta_1,\beta_2)\in\mathcal T_m\times \mathcal T_n}\frac{1_{\hat K^\prime_{\beta_1,\beta_2}}(w_1,w_2)}{\mu(\hat K^\prime_{\beta_1,\beta_2})}\int_{\hat K^\prime_{\beta_1,\beta_2}}|f(z_1,z_2)|\mu(z_1,z_2)dV(z_1,z_2).
\end{equation}
\end{de}
We set $\langle f\rangle_{Q,\mu}:=\frac{\int_Q|f|d\mu}{\mu(Q)}$, then we also have:
\begin{equation}\label{2.10}
\mathcal M_{\mathcal T^\prime_{m,n},\mu}f(w_1,w_2)=\sup_{(\beta_1,\beta_2)\in\mathcal T_m\times \mathcal T_n}{1_{\hat K^\prime_{\beta_1,\beta_2}}(w_1,w_2)}\langle f\rangle_{\hat K^\prime_{\beta_1,\beta_2},\mu}.
\end{equation}
We have the following $L^p$ regularity result for $\mathcal M_{\mathcal T^\prime_{m,n},\mu}$.
\begin{lem}
\label{l:2.7}
	Let $\mu(z_1,z_2)$ the same as in Theorem \ref{t:main}, then $\mathcal M_{\mathcal T^\prime_{m,n},\mu}$ is bounded on $L^p(\mathbb H, \mu)$ for $1<p\leq\infty$. Moreover, $\|\mathcal M_{\mathcal T^\prime_{m,n},\mu}\|_{L^p(\mathbb H,\mu)}\lesssim (p/(p-1))^{2}$ for $1<p<\infty$.
\end{lem}
\begin{proof}When $p=\infty$, the boundedness of $\mathcal M_{\mathcal T^\prime_{m,n},\mu}$ is obvious. We turn to the case $1<p<\infty$. Set $\mu^\prime_2(w_2):=|w_2|^2\mu_2(w_2)$.
Using the biholomorphism $h: (w_1,w_2)\mapsto ({w_1}{w_2},w_2)$ from  $\mathbb D\times \mathbb D^*$ onto $\mathbb H$, we transform $\mathcal M_{\mathcal T^\prime_{m,n},\mu}$ into the following maximal function on $\mathbb D\times \mathbb D^*$:
\begin{equation}\label{2.101}
\mathcal M_{\mathcal T_{m,n},\mu}f(w_1,w_2):=\sup_{(\beta_1,\beta_2)\in\mathcal T_m\times \mathcal T_n}\frac{1_{\hat K_{\beta_1}}(w_1)1_{\hat K_{\beta_2}}(w_2)}{\mu_1(\hat K_{\beta_1})\mu^\prime_2(\hat K_{\beta_2})}\int_{\hat K_{\beta_1,\beta_2}}|f(z_1,z_2)|\mu_1(z_1)\mu^\prime_2(z_2)dV(z_1,z_2),
\end{equation} and  it suffices to show that $\mathcal M_{\mathcal T_{m,n},\mu}$ is $L^p$ bounded on $L^p(\mathbb D\times \mathbb D^*,|w_2|^2\mu\circ h)$ for $1<p\leq \infty$.
Defining the following two 1-parameter maximal functions: \begin{align}\mathcal M_{\mathcal T_m,\mu_1}f(w_1,w_2)&:=\sup_{\beta_1\in\mathcal T_m}\frac{1_{\hat K_{\beta_1}}(w_1)}{\mu_1(\hat K_{\beta_1})}\int_{\hat K_{\beta_1}}|f(z_1,w_2)|\mu_1(z_1)dV(z_1);\\\mathcal M_{\mathcal T_n,\mu^\prime_2}f(w_1,w_2)&:=\sup_{\beta_2\in\mathcal T_n}\frac{1_{\hat K_{\beta_2}}(w_2)}{\mu^\prime_2(\hat K_{\beta_2})}\int_{\hat K_{\beta_2}}|f(w_1,z_2)|\mu^\prime_2(z_2)dV(z_2),
\end{align}
we obtain that  $\mathcal M_{\mathcal T_{m,n},\mu}f\leq \mathcal M_{\mathcal T_m,\mu_1}\circ \mathcal M_{\mathcal T_n,\mu^\prime_2}f$. By Fubini's Theorem, it is enough to show that $\mathcal M_{\mathcal T_m,\mu_1}$ is bounded on $L^p(\mathbb D,\mu_1dV)$ and $\mathcal M_{\mathcal T_n,\mu^\prime_2}$ is bounded on $L^p(\mathbb D,\mu^\prime_2dV)$. Here we show the $L^p$ boundedness of $\mathcal M_{\mathcal T_m,\mu_1}$. The boundedness of $\mathcal M_{\mathcal T_n,\mu^\prime_2}$ follows from an analogous argument.

Note that $\mathcal M_{\mathcal T_m,\mu_1}$ is bounded on $L^\infty(\mathbb D,\mu_1)$. By interpolation, the weak-type (1,1) estimate
\begin{equation}\label{2.141}
\mu_1(\left\{z\in \mathbb D:\mathcal M_{\mathcal T_m,\mu_1}f(z)>\lambda\right\})\lesssim \frac{\|f\|_{L^1(\mathbb D,\mu_1)}}{\lambda}
\end{equation}
is sufficient to finish the proof. For a point $w\in \left\{z\in \mathbb D:\mathcal M_{\mathcal T_m,\mu_1}f(z)>\lambda\right\}$, there exists a unique maximal tent $\hat K_\alpha$ that contains $w$ and satisfies:
\begin{equation}
\frac{1_{\hat K_{\alpha}}(w)}{\mu_1(\hat K_{\alpha})}\int_{\hat K_{\alpha}}|f(z)|\mu_1(z)dV(z)>\frac{\lambda}{2}.
\end{equation}
Let $\mathcal A_\lambda$ be the set of indices of all such maximal tents $\hat K_\alpha$. The union of these maximal tents covers the set $\left\{z\in \mathbb D:\mathcal M_{\mathcal T_m,\mu_1}f(z)>\lambda\right\}$. Since the tents $\hat K_\alpha$ are maximal, they are also pairwise disjoint and hence 
\begin{equation*}
\mu_1(\left\{z\in \mathbb D:\mathcal M_{\mathcal T_m,\mu_1}f(z)>\lambda\right\})\leq \sum_{\alpha\in \mathcal A_\lambda}\mu_1(\hat K_\alpha)\leq \sum_{\alpha\in \mathcal A_\lambda}\frac{2}{\lambda}\int_{ \hat{K}_\alpha}f(z)\mu_1(z)dV(z)\leq\frac{2\|f\|_{L^1(\mathbb D,\mu_1)}}{\lambda}.
\end{equation*}
Thus inequality (\ref{2.141}) holds and $\mathcal M_{\mathcal T_m,\mu_1}$ is weak-type (1,1). Using a standard argument for the Hardy-Littlewood maximal function,  we further have $$\|\mathcal M_{\mathcal T_m,\mu_1}\|_{L^p(\mathbb D\times \mathbb D^*,|w_2|^2\mu\circ h)}\lesssim \frac{p}{p-1}.$$ Since the same inequality holds for $\mathcal M_{\mathcal T_n,\mu_2^\prime}$, 
\begin{align*}\|\mathcal M_{\mathcal T^\prime_{m,n},\mu}\|_{L^p(\mathbb H,\mu)}=&\|\mathcal M_{\mathcal T_{m,n},\mu}\|_{L^p(\mathbb D\times \mathbb D^*,|w_2|^2\mu\circ h)}\\\leq&\|\mathcal M_{\mathcal T_m,\mu_1}\circ \mathcal M_{\mathcal T_n,\mu_2^\prime}\|_{L^p(\mathbb D\times \mathbb D^*,|w_2|^2\mu\circ h)}\lesssim \left(\frac{p}{p-1}\right)^{2}.\end{align*}
\end{proof}

Finally, we define two operators $Q$ and $Q^+$. Let $p^\prime$ be the conjugate index of $p$. We set
\begin{align}
Q(f)(z_1,z_2)&=\int_{\mathbb H} \frac{1}{\pi^2z_2(1-\frac{z_1\bar w_1}{z_2 \bar w_2})^2(1-z_2\bar w_2)^2}f(w_1,w_2)dV(w_1,w_2),\\Q^+(f)(z_1,z_2)&=\int_{\mathbb H} \frac{1}{\pi^2|z_2||1-\frac{z_1\bar w_1}{z_2 \bar w_2}|^2|1-z_2\bar w_2|^2}f(w_1,w_2)dV(w_1,w_2).
\end{align}
It is clear that $P=QM_{1/\bar w_2}$ and $P^+=Q^+M_{1/|w_2|}$. Moreover, the weighted $L^p$ norm of the projection, $\|P^+:L^p(\mathbb H, \mu dV)\to L^p(\mathbb H, \mu dV) \|$, is equal to the weighted norm of $Q^+M_{\nu}$ acting between two different weighted $L^p$ spaces.
\begin{lem}
\label{l:2.8}
Let $\mu$ be a weight on the Hartogs triangle. Set $\nu:=\mu^{\frac{-p^\prime}{p}}|w_2|^{-p^\prime}$. Then
\begin{align}\label{2.190}
&\|P:L^p(\mathbb H, \mu dV)\to L^p(\mathbb H, \mu dV) \|=\|QM_\nu:L^p(\mathbb H,\nu dV)\to L^p(\mathbb H,\mu dV)\|;\\&
\label{2.191}
\|P^+:L^p(\mathbb H, \mu dV)\to L^p(\mathbb H, \mu dV) \|=\|Q^+M_\nu:L^p(\mathbb H,\nu dV)\to L^p(\mathbb H,\mu dV)\|.
\end{align}
\begin{proof} We show (\ref{2.191}) here as the proof for (\ref{2.190}) is similar.
Given $f\in L^p(\mathbb H,\mu)$, we have
\begin{equation}
\int_{\mathbb H}|f|^p\mu dV(w_1,w_2)=\int_{\mathbb H}\left|\frac{f}{w_2}\right|^p|w_2|^p \mu dV(w_1,w_2)=\int_{\mathbb H}\left|M_\frac{1}{|w_2|}f\right|^p|w_2|^{p}\mu dV(w_1,w_2).
\end{equation}
Thus $\|f\|_{L^p(\mu dV)}=\|M_{1/|w_2|} f\|_{L^p(\mu|w_2|^{p}dV)}$ and $$\|P^+:L^p(\mathbb H, \mu dV)\to L^p(\mathbb H, \mu dV) \|=\|Q^+:L^p(\mathbb H,|w_2|^p\mu dV)\to L^p(\mathbb H,\mu dV)\|.$$
We claim further that for $f\in L^p(\mathbb H,|w_2|^p\mu dV)$,  $\|f\|_{L^p(|w_2|^p\mu dV)}=\|M_{1/\nu}f\|_{L^p(\nu dV)}$. Then (\ref{2.191}) holds.
Recall that $\nu:=\mu^{\frac{-p^\prime}{p}}|w_2|^{-p^\prime}$. We have
\begin{equation*}
\int_{\mathbb H}\left|\frac{f}{\nu}\right|^p\nu dV=\int_{\mathbb H}|{f}|^p\nu^{1-p} dV=\int_{\mathbb H}|{f}|^p(\mu^{-\frac{p^\prime}{p}}|w_2|^{-p^\prime})^{1-p} dV=\int_{\mathbb H}|{f}|^p|w_2|^p\mu dV.
\end{equation*}
Hence the claim is shown and the proof is complete.
\end{proof}
\end{lem}
\section{Proof of Theorem \ref{t:main}}
It is sufficient to prove that inequality (\ref{2.102}) holds. 

{\subsection{Proof for the upper bound} For the upper bound inequality $$\|P^+\|_{L^p(\mathbb H,\mu dV)}\lesssim ([\mu,\nu]^{0,0}_p)^{\frac{1}{p}}+pp^\prime([\mu,\nu]^{1,0}_p+[\mu,\nu]^{0,1}_p)+(pp^\prime)^2([\mu,\nu]^{1,1}_p)^{\max\{1,\frac{1}{p-1}\}},$$
we first consider the case $p\geq2$. The case $1<p<2$ will follow from a duality argument. 

Recall the tree structure $\{\mathcal T^\prime_{m,n}\}_{m=1}^M$ and the dyadic tent $\{\hat K^\prime_{\beta_1,\beta_2}\}$ from (\ref{2.7}) and (\ref{2.8}). Set the measure $du:=|w_2|^{-2}dV$. By Lemma \ref{l:2.2} and the inequality (\ref{2.6}), there is a finite collection $M$ such that for $(z_1,z_2)$ and $(w_1,w_2)$ in $\mathbb H$, there exists  $\hat K_{\beta_1}$and $\hat K_{\beta_1}$ such that
\begin{align}
\left|1-\frac{z_1\bar w_1}{z_2 \bar w_2}\right|^{-2}|1-z_2\bar w_2|^{-2}\approx& |\hat K_{\beta_1}|^{-1}|\hat K_{\beta_2}|^{-1}\nonumber\\\leq&\sum_{m,n=1}^{M}\sum_{\gamma\in \mathcal T_m,\eta\in \mathcal T_n}\frac{1_{\hat K_\gamma\times\hat K_\eta}(z_1/z_2,z_2)1_{\hat K_\gamma\times\hat K_\eta}(w_1/w_2,w_2)}{|\hat K_\gamma\times\hat K_\eta|}\nonumber\\=&\sum_{m,n=1}^{M}\sum_{(\gamma,\eta)\in \mathcal T^\prime_{m,n}}\frac{1_{\hat K^\prime_{\gamma,\eta}}(z_1,z_2)1_{\hat K^\prime_{\gamma,\eta}}(w_1,w_2)}{u(\hat K^\prime_{\gamma,\eta})}.
\end{align}
Applying this inequality to the operator $Q^+M_\nu$ yields
\begin{align}
&\left|Q^+M_\nu f(z_1,z_2)\right|\nonumber\\=&\left|\int_{\mathbb H}\frac{|z_2|^{-1}M_\nu f(w_1,w_2)}{\pi^2|1-\frac{z_1\bar w_1}{z_2 \bar w_2}|^2|1-z_2\bar w_2|^2} dV(w_1,w_2)\right|
\nonumber\\\lesssim&\int_{\mathbb H}\sum_{m,n=1}^{M}\sum_{(\gamma,\eta)\in \mathcal T^\prime_{m,n}}\frac{1_{\hat K^\prime_{\gamma,\eta}}(z_1,z_2)1_{\hat K^\prime_{\gamma,\eta}}(w_1,w_2)\left|M_\nu f(w_1,w_2)\right|}{|z_2|u(\hat K^\prime_{\gamma,\eta})}dV(w_1,w_2)\nonumber\\=&\sum_{m,n=1}^{M}\sum_{(\gamma,\eta)\in \mathcal T^\prime_{m,n}}\frac{1_{\hat K^\prime_{\gamma,\eta}}(z_1,z_2)}{|z_2|}\langle  f\nu|w_2|^2\rangle^{du}_{\hat K^\prime_{\gamma,\eta}}
\nonumber\\=&I_{0,0}+I_{0,1}+I_{1,0}+I_{1,1},
\end{align}
where
\begin{align}
I_{0,0}&=\sum_{m,n=1}^{M}\frac{1_{\hat K^\prime_{0,0}}(z_1,z_2)}{|z_2|}\langle  f\nu|w_2|^2\rangle^{du}_{K^\prime_{0,0}}=M^2\frac{1_{\mathbb H}(z_1,z_2)}{|z_2|}\langle  f\nu|w_2|^2\rangle^{du}_{\mathbb H};\\
I_{1,0}&=\sum_{m,n=1}^{M}\sum_{(\gamma,0)\in \mathcal T^\prime_{m,n}}\frac{1_{\hat K^\prime_{\gamma,0}}(z_1,z_2)}{|z_2|}\langle  f\nu|w_2|^2\rangle^{du}_{\hat K^\prime_{\gamma,0}};\\
I_{0,1}&=\sum_{m,n=1}^{M}\sum_{(0,\eta)\in \mathcal T^\prime_{m,n}}\frac{1_{\hat K^\prime_{0,\eta}}(z_1,z_2)}{|z_2|}\langle  f\nu|w_2|^2\rangle^{du}_{\hat K^\prime_{0,\eta}};\\
I_{1,1}&=\sum_{m,n=1}^{M}\sum_{\substack{(\gamma,\eta)\in \mathcal T^\prime_{m,n}\\\gamma,\eta\neq 0}}1_{\hat K^\prime_{\gamma,\eta}}(z_1,z_2)|z_2|^{-1}\langle  f\nu|w_2|^2\rangle^{du}_{\hat K^\prime_{\gamma,\eta}}.
\end{align}
Set \begin{align*}Q^{0,0}_{m,n,\nu}f(z_1,z_2)&:=\frac{1_{\mathbb H}(z_1,z_2)}{|z_2|}\langle  f\nu|w_2|^2\rangle^{du}_{\mathbb H};\\Q^{1,0}_{m,n,\nu}f(z_1,z_2)&:=\sum_{(\gamma,0)\in \mathcal T^\prime_{m,n}}1_{\hat K^\prime_{\gamma,0}}(z_1,z_2)|z_2|^{-1}\langle  f\nu|w_2|^2\rangle^{du}_{\hat K^\prime_{\gamma,0}};\\Q^{0,1}_{m,n,\nu}f(z_1,z_2)&:=\sum_{(0,\eta)\in \mathcal T^\prime_{m,n}}1_{\hat K^\prime_{0,\eta}}(z_1,z_2)|z_2|^{-1}\langle  f\nu|w_2|^2\rangle^{du}_{\hat K^\prime_{0,\eta}};\\Q^{1,1}_{m,n,\nu}f(z_1,z_2)&:=\sum_{\substack{(\gamma,\eta)\in \mathcal T^\prime_{m,n}\\\gamma,\eta\neq 0}}1_{\hat K^\prime_{\gamma,\eta}}(z_1,z_2)|z_2|^{-1}\langle  f\nu|w_2|^2\rangle^{du}_{\hat K^\prime_{\gamma,\eta}}.\end{align*} Then it suffices to estimate the $L^p$ norm for each $Q^{i,j}_{m,n,\nu}$.  The proof given below uses the idea of how to prove the linear bound for sparse operators in the weighted theory of harmonic analysis, see for example \cite{Moen2012} and \cite{Lacey2017}.  

We first consider $Q^{0,0}_{m,n,\nu}$. For arbitrary $g\in L^{p^\prime}(\mathbb H,\mu)$,
\begin{align}\label{3.0}
&\left\langle Q^{0,0}_{m,n,\nu} f(z_1,z_2), g(z_1,z_2)\mu\right\rangle\nonumber\\=&\int_{\mathbb H} Q^{0,0}_{m,n,\nu} f(z_1,z_2)g(z_1,z_2)\mu dV(z_1,z_2)\nonumber\\=&\int_{\mathbb H}1_{\mathbb H}(z_1,z_2)|z_2|^{-1}\langle  f\nu|w_2|^2\rangle^{du}_{\mathbb H}g(z_1,z_2)\mu dV(z_1,z_2)\nonumber\\=&(u(\mathbb H))^{-1}\int_{\mathbb H} f(z_1,z_2)\nu dV(z_1,z_2)\int_{\mathbb H}g(z_1,z_2)|z_2|^{-1}\mu dV(z_1,z_2)\nonumber\\\leq&(u(\mathbb H))^{-1}\left(\int_{\mathbb H}\nu dV\right)^{\frac{p-1}{p}}\|f\|_{L^p(\mathbb H,\nu)}\left(\int_{\mathbb H}|z_2|^{-p}\mu dV\right)^\frac{1}{p}\|g\|_{L^{p^\prime}(\mathbb H,\mu)}\nonumber\\=&\left(\langle |w_2|^{2-p}\mu\rangle^{du}_{\mathbb H}\left(\langle\nu|w_2|^2 \rangle^{du}_{\mathbb H}\right)^{p-1}\right)^{\frac{1}{p}}  \|f\|_{L^p(\mathbb H,\nu)}\|g\|_{L^{p^\prime}(\mathbb H,\mu)}.
\end{align}
Therefore \begin{align}\label{3.80}\|Q^{0,0}_{m,n,\nu}\|_{L^p(\mathbb H,\nu)\to L^p(\mathbb H,\mu)}\leq \left(\langle |w_2|^{2-p}\mu\rangle^{du}_{\mathbb H}\left(\langle\nu|w_2|^2 \rangle^{du}_{\mathbb H}\right)^{p-1}\right)^{\frac{1}{p}}=([\mu,\nu]_p^{0,0})^{\frac{1}{p}}.\end{align}

We turn to $Q^{1,1}_{m,n,\nu}$. For arbitrary $g\in L^{p^\prime}(\mathbb H,\mu)$,
\begin{align}\label{3.1}
&\left\langle Q^{1,1}_{m,n,\nu} f(z_1,z_2), g(z_1,z_2)\mu\right\rangle\nonumber\\=&\int_{\mathbb H} Q^{1,1}_{m,n,\nu} f(z_1,z_2)g(z_1,z_2)\mu dV(z_1,z_2)\nonumber\\=&\int_{\mathbb H}\sum_{\substack{(\gamma,\eta)\in \mathcal T^\prime_{m,n}\\\gamma,\eta\neq 0}}1_{\hat K^\prime_{\gamma,\eta}}(z_1,z_2)|z_2|^{-1}\langle  f\nu|w_2|^2\rangle^{du}_{\hat K^\prime_{\gamma,\eta}}g(z_1,z_2)\mu dV(z_1,z_2)\nonumber\\=&\sum_{\substack{(\gamma,\eta)\in \mathcal T^\prime_{m,n}\\\gamma,\eta\neq 0}}\langle  f\nu|w_2|^2\rangle^{du}_{\hat K^\prime_{\gamma,\eta}}\int_{\hat K^\prime_{\gamma,\eta}}g(z_1,z_2)|z_2|^{-1}\mu dV(z_1,z_2)\nonumber\\=&\sum_{\substack{(\gamma,\eta)\in \mathcal T^\prime_{m,n}\\\gamma,\eta\neq 0}}\langle  f\rangle^{\nu dV}_{\hat K^\prime_{\gamma,\eta}}\langle\nu|w_2|^2 \rangle^{du}_{\hat K^\prime_{\gamma,\eta}}\langle g|w_2|^{p-1}\rangle^{|w_2|^{2-p}\mu du}_{\hat K^\prime_{\gamma,\eta}} \langle |w_2|^{2-p}\mu\rangle^{du}_{\hat K^\prime_{\gamma,\eta}}u(\hat K_{\gamma,\eta})\nonumber\\=&\sum_{\substack{(\gamma,\eta)\in \mathcal T^\prime_{m,n}\\\gamma,\eta\neq 0}}\left(\langle\nu|w_2|^2 \rangle^{du}_{\hat K^\prime_{\gamma,\eta}}\right)^{p-1}  \langle |w_2|^{2-p}\mu\rangle^{du}_{\hat K^\prime_{\gamma,\eta}}\langle  f\rangle^{\nu dV}_{\hat K^\prime_{\gamma,\eta}}\langle g|w_2|^{p-1}\rangle^{|w_2|^{-p}\mu dV}_{\hat K^\prime_{\gamma,\eta}}u(\hat K^\prime_{\gamma,\eta})\left(\langle\nu|w_2|^2 \rangle^{du}_{\hat K^\prime_{\gamma,\eta}}\right)^{2-p}\nonumber\\\leq&[\mu,\nu]^{1,1}_p\sum_{\substack{(\gamma,\eta)\in \mathcal T^\prime_{m,n}\\\gamma,\eta\neq 0}}\langle  f\rangle^{\nu dV}_{\hat K^\prime_{\gamma,\eta}}\langle g|w_2|^{p-1}\rangle^{|w_2|^{-p}\mu dV}_{\hat K^\prime_{\gamma,\eta}}\left(u(\hat K^\prime_{\gamma,\eta})\right)^{p-1}\left(\nu(\hat K^\prime_{\gamma,\eta})\right)^{2-p}.
\end{align}
Recall from Lemma \ref{l:2.1} that $|\hat K_\alpha|\approx |K_\alpha|$ for the tree structure $\mathcal T$ with Lebesgue measure $\sigma$ on the unit disc. Hence for the induced tree structure $\mathcal T^\prime_{m,n}$ with the induced weighted measure $u$ on the Hartogs triangle, we also have  $u(\hat K^\prime_{\gamma,\eta})\approx u(K^\prime_{\gamma,\eta})$. The facts that $p\geq 2$ and $K^\prime_{\gamma,\eta}\subseteq \hat K^\prime_{\gamma,\eta}$ gives the inequality
$\left(\nu(\hat K^\prime_{\gamma,\eta})\right)^{2-p}\leq\left(\nu( K^\prime_{\gamma,\eta})\right)^{2-p}$. Combining these facts, we have
\begin{equation}
\left(u(\hat K^\prime_{\gamma,\eta})\right)^{p-1}\left(\nu(\hat K^\prime_{\gamma,\eta})\right)^{2-p}\lesssim \left(u( K^\prime_{\gamma,\eta})\right)^{p-1}\left(\nu( K^\prime_{\gamma,\eta})\right)^{2-p}.
\end{equation}
By H\"older's inequality, $$u(K^\prime_{\gamma,\eta})\leq \left(\nu(K^\prime_{\gamma,\eta})\right)^{\frac{1}{p^\prime}}\left(\int_{K^\prime_{\gamma,\eta}}|w_2|^{-p}\mu dV\right)^{\frac{1}{p}}.$$
Therefore,
\begin{equation}
\left(u( K^\prime_{\gamma,\eta})\right)^{p-1}\left(\nu( K^\prime_{\gamma,\eta})\right)^{2-p}\leq \left(\nu( K^\prime_{\gamma,\eta})\right)^{\frac{1}{p}}\left(\int_{K^\prime_{\gamma,\eta}}|w_2|^{-p}\mu dV\right)^{\frac{1}{p^\prime}}.
\end{equation} 
Applying these inequalities to the last line of (\ref{3.1}), we have
\begin{align}
&[\mu,\nu]^{1,1}_p\sum_{(\gamma,\eta)\in \mathcal T^\prime_{m,n}}\langle  f\rangle^{\nu dV}_{\hat K^\prime_{\gamma,\eta}}\langle g|w_2|^{p-1}\rangle^{|w_2|^{-p}\mu dV}_{\hat K^\prime_{\gamma,\eta}}\left(u(\hat K^\prime_{\gamma,\eta})\right)^{p-1}\left(\nu(\hat K^\prime_{\gamma,\eta})\right)^{2-p}\nonumber\\\lesssim&[\mu,\nu]^{1,1}_p\sum_{(\gamma,\eta)\in \mathcal T^\prime_{m,n}}\langle  f\rangle^{\nu dV}_{\hat K^\prime_{\gamma,\eta}}\langle g|w_2|^{p-1}\rangle^{|w_2|^{-p}\mu dV}_{\hat K^\prime_{\gamma,\eta}}\left(\nu( K^\prime_{\gamma,\eta})\right)^{\frac{1}{p}}\left(\int_{K^\prime_{\gamma,\eta}}|w_2|^{-p}\mu dV\right)^{\frac{1}{p^\prime}}.
\end{align}
Applying H\"older's inequality again to the sum above yields:
\begin{align}\label{3.3}
&\sum_{(\gamma,\eta)\in \mathcal T^\prime_{m,n}}\langle  f\rangle^{\nu dV}_{\hat K^\prime_{\gamma,\eta}}\langle g|w_2|^{p-1}\rangle^{|w_2|^{-p}\mu dV}_{\hat K^\prime_{\gamma,\eta}}\left(\nu( K^\prime_{\gamma,\eta})\right)^{\frac{1}{p}}\left(\int_{K^\prime_{\gamma,\eta}}|w_2|^{-p}\mu dV\right)^{\frac{1}{p^\prime}}\nonumber
\\\leq&\left(\sum_{(\gamma,\eta)\in \mathcal T^\prime_{m,n}}\left(\langle  f\rangle^{\nu dV}_{\hat K^\prime_{\gamma,\eta}}\right)^{p}\nu(K^\prime_{\gamma,\eta})\right)^{\frac{1}{p}}\left(\sum_{(\gamma,\eta)\in \mathcal T^\prime_{m,n}}\left(\langle g|w_2|^{p-1}\rangle^{|w_2|^{-p}\mu dV}_{K^\prime_{\gamma,\eta}}\right)^{p^\prime}\int_{K^\prime_{\gamma,\eta}}|w_2|^{-p}\mu dV\right)^{\frac{1}{p^\prime}}.
\end{align}
By the disjointness of $K^\prime_{\gamma,\eta}$ and Lemma \ref{l:2.7}, we have 
\begin{equation}\label{3.7}
\sum_{(\gamma,\eta)\in \mathcal T^\prime_{m,n}}\left(\langle  f\rangle^{\nu dV}_{\hat K^\prime_{\gamma,\eta}}\right)^{p}\nu( K^\prime_{\gamma,\eta})\leq \int_{\mathbb H} (\mathcal M_{\mathcal T^\prime_{m,n},\nu}f)^p\nu dV\leq (p^\prime)^{2p}\|f\|^{2p}_{L^p(\mathbb H,\nu dV)}.
\end{equation}
Note that $\|g|w_2|^{p-1}\|_{L^{p^\prime}(\mathbb H,|w_2|^{-p}\mu dV)}=\|g\|_{L^{p^\prime}(\mathbb H,\mu dV)}$. A similar argument using the maximal function $\mathcal M_{\mathcal T^\prime_{m,n},|w_2|^{-p}\mu}$ will also give the inequality 
\begin{equation}\label{3.8}
\sum_{(\gamma,\eta)\in \mathcal T^\prime_{m,n}}\left(\langle g|w_2|^{p-1}\rangle^{|w_2|^{-p}\mu dV}_{\hat K^\prime_{\gamma,\eta}}\right)^p\int_{K^\prime_{\gamma,\eta}}|w_2|^{-p}\mu dV\leq (p)^{2p\prime}\|g\|^{p^\prime}_{L^{p^\prime}(\mathbb H,\mu dV)}.
\end{equation}
Substituting (\ref{3.7}) and (\ref{3.8}) into (\ref{3.3}) and (\ref{3.1}) finally yields
\begin{equation}
\left\langle Q^{1,1}_{m,n,\nu} f, g\mu\right\rangle\lesssim [\mu,\nu]^{1,1}_p (pp^\prime)^{2}\|f\|_{L^p(\mathbb H,\nu dV)} \|g\|_{L^{p^\prime}(\mathbb H,\mu dV)}.
\end{equation}
Therefore $\|Q^{1,1}_{m,n,\nu} \|_{L^p(\mathbb H,\nu dV)\to L^p(\mathbb H,\mu dV)}\lesssim(pp^\prime)^2[\mu,\nu]^{1,1}_p$.

For the case $1<p<2$ and we claim that
\begin{equation}
\left \langle Q^{1,1}_{m,n,\nu} f, g\mu\right\rangle\lesssim([\mu,\nu]^{1,1}_p)^{\frac{1}{p-1}}\|f\|_{L^p(\mathbb H,\nu dV)} \|g\|_{L^{p^\prime}(\mathbb H,\mu dV)},
\end{equation}
for all $f\in L^p(\mathbb H,\nu dV)$ and $g\in L^{p^\prime}(\mathbb H,\mu dV)$. By the definition of $Q^{1,1}_{m,n,\nu}$,
\begin{align}\label{3.12}
\left \langle Q^{1,1}_{m,n,\nu} f, g\mu\right\rangle&=\left\langle\sum_{(\gamma,\eta)\in \mathcal T^\prime_{m,n}}1_{\hat K^\prime_{\gamma,\eta}}(w_1,w_2)|w_2|^{-1}\langle  f\nu|w_2|^2\rangle^{du}_{\hat K^\prime_{\gamma,\eta}},g\mu\right\rangle\nonumber\\&=\sum_{(\gamma,\eta)\in \mathcal T^\prime_{m,n}}\left\langle 1_{\hat K^\prime_{\gamma,\eta}}(w_1,w_2)\langle  f\nu|w_2|^2\rangle^{du}_{\hat K^\prime_{\gamma,\eta}},g|w_2|^{-1}\mu\right\rangle\nonumber\\&=\sum_{(\gamma,\eta)\in \mathcal T^\prime_{m,n}}\langle  f\nu|w_2|^2\rangle^{du}_{K^\prime_{\gamma,\eta}}\langle g|w_2|\mu\rangle^{du}_{\hat K^\prime_{\gamma,\eta}}u(\hat K^\prime_{\gamma,\eta})\nonumber
\\&=\sum_{(\gamma,\eta)\in \mathcal T^\prime_{m,n}}\left\langle 1_{\hat K^\prime_{\gamma,\eta}}(w_1,w_2)|w_2|^{-1}\langle  g|w_2|^{p-1}|w_2|^{-p}\mu|w_2|^2\rangle^{du}_{\hat K^\prime_{\gamma,\eta}}|w_2|,f\nu\right\rangle
\nonumber\\&=\left\langle M_{|z_2|}Q^{1,1}_{m,n,|w_2|^{-p}\mu}(g|w_2|^{p-1}),f\nu\right\rangle.
\end{align}
Set $h=g|w_2|^{p-1}$ and $\psi=|w_2|^{-p}\mu$. Then $\|h\|_{L^{p^\prime}(\mathbb H,\psi dV)}=\|g\|_{L^{p^\prime}(\mathbb H,\mu dV)}$. Setting the weight $\omega$ to satisfies $|w_2|^{-p}\omega^{\frac{-p}{-p^\prime}}=\psi$, we have $\omega=\mu^\frac{p^\prime}{p}=\nu|z_2|^{p^\prime}$. Replacing $p$ by $p^\prime$, $\mu$ by $\omega$, and $\nu$ by $\psi$ and going through same argument for the case $p\geq2$ yields that 
\begin{align}
\|M_{|z_2|}Q^{1,1}_{m,n,|w_2|^{-p}\mu} \|_{L^{p^\prime}(\mathbb H,\nu dV)}&=\|Q^{1,1}_{m,n,|w_2|^{-p}\mu}\|_{L^{p^\prime}(\mathbb H,|w_2|^{p^\prime}\nu dV)} \nonumber\\&\lesssim(pp^\prime)^2\sup_{(\gamma,\eta)\in \mathcal T^\prime_{m,n}}\left(\langle\mu|w_2|^{2-p} \rangle^{du}_{\hat K^\prime_{\gamma,\eta}}\right)^{p^\prime-1}  \langle |w_2|^{2-p^\prime}\nu|z_2|^{p^\prime}\rangle^{du}_{\hat K^\prime_{\gamma,\eta}}
\nonumber\\&=(pp^\prime)^2\left(\sup_{(\gamma,\eta)\in \mathcal T^\prime_{m,n}}\langle\mu|w_2|^{2-p} \rangle^{du}_{\hat K^\prime_{\gamma,\eta}}\left(\langle |w_2|^{2}\nu\rangle^{du}_{\hat K^\prime_{\gamma,\eta}}\right)^{p-1}\right)^{\frac{1}{p-1}}\nonumber\\&=(pp^\prime)^2([\mu,\nu]^{1,1}_p)^{\frac{1}{p-1}}.
\end{align}
 Thus we have $$\left \langle Q^{1,1}_{m,n,\nu} f, g\mu\right\rangle\lesssim(pp^\prime)^2([\mu,\nu]^{1,1}_p)^\frac{1}{p-1}\|g\|_{L^{p^\prime}(\mathbb H,\mu dV)}\|f\|_{L^p(\mathbb H,\nu dV)},$$ and $$\| Q^{1,1}_{m,n,\nu}\|_{L^p(\mathbb H,\mu dV)}\lesssim (pp^\prime)^2([\mu,\nu]^{1,1}_p)^{\frac{1}{p-1}}.$$
 Combining the results for $1<p< 2$ and $p\geq2$ gives: \begin{align}\label{3.200}\|  Q^{1,1}_{m,n,\nu}\|_{L^p(\mathbb H,\nu dV)\to L^p(\mathbb H,\mu dV)}\lesssim (pp^\prime)^2([\mu,\nu]^{1,1}_p)^{\max\{1,\frac{1}{p-1}\}}.\end{align}
 
To estimate $\|Q^{1,0}_{m,n,\nu}\|_{L^p(\mathbb H,\nu dV)\to L^p(\mathbb H,\mu dV)}$, we combine the above arguments for $Q^{0,0}_{m,n,\nu}$ and $Q^{1,1}_{m,n,\nu}$. For arbitrary $g\in L^{p^\prime}(\mathbb H,\mu)$,
 \begin{align}\label{3.2}
 &\left\langle Q^{1,0}_{m,n,\nu} f(z_1,z_2), g(z_1,z_2)\mu\right\rangle\nonumber\\=&\int_{\mathbb H} Q^{1,0}_{m,n,\nu} f(z_1,z_2)g(z_1,z_2)\mu dV(z_1,z_2)\nonumber\\=&\int_{\mathbb H}\sum_{(\gamma,0)\in \mathcal T^\prime_{m,n}}1_{\hat K^\prime_{\gamma,0}}(z_1,z_2)|z_2|^{-1}\langle  f\nu|w_2|^2\rangle^{du}_{\hat K^\prime_{\gamma,0}} g(z_1,z_2)\mu dV(z_1,z_2)\nonumber\\=&\sum_{(\gamma,0)\in \mathcal T^\prime_{m,n}}(u(\hat K^\prime_{\gamma,0}))^{-1}\int_{ \hat K^\prime_{\gamma,0}}|f(z_1,z_2)|\nu dV(z_1,z_2)\int_{ \hat K^\prime_{\gamma,0}}|z_2|^{-1}|g(z_1,z_2)|\mu dV(z_1,z_2)\nonumber\\\approx&\sum_{(\gamma,0)\in \mathcal T^\prime_{m,n}}|\hat K_{\gamma}|^{-1}\int_{ \hat K^\prime_{\gamma,0}}|f(z_1,z_2)|\nu dV(z_1,z_2)\int_{ \hat K^\prime_{\gamma,0}}|g(z_1,z_2)||z_2|^{-1}\mu dV(z_1,z_2).
 \end{align}
Recall that $\mu(z_1,z_2)=\mu_1(z_1/z_2)\mu_2(z_2)$. There holds $\nu(z_1,z_2)=\nu_1(z_1/z_2)\nu_2(z_2)$ by the definition of $\nu$. Hence \begin{align}\label{3.20}
&\int_{ \hat K^\prime_{\gamma,0}}|f(z_1,z_2)|\nu dV(z_1,z_2)\nonumber\\=&\int_{ \hat K_{\gamma}\times \mathbb  D}|f(z_2t,z_2)|\nu_1(t)\nu_2(z_2) |z_2|^2dV(t,z_2)\nonumber\\\leq&\left(\int_{\mathbb  D}\left|\int_{\hat K_\gamma}|f(z_2t,z_2)|\nu_1(t)dV(t)\right|^p\nu_2(z_2) |z_2|^2dV(z_2)\right)^{\frac{1}{p}}\left(\int_{\mathbb D}\nu_2(z_2) |z_2|^2dV(z_2)\right)^{\frac{1}{p^\prime}}
\nonumber\\\approx&\nu_1(\hat K_\gamma)\left(\int_{\mathbb  D}\frac{|\int_{\hat K_\gamma}|f(z_2t,z_2)|\nu_1(t)dV(t)|^p}{(\nu_1(\hat K_\gamma))^p}\nu_2(z_2) |z_2|^2dV(z_2)\right)^{\frac{1}{p}}\left(\int_{\mathbb D}\nu_2(z_2) |z_2|^2dV(z_2)\right)^{\frac{1}{p^\prime}}.
\end{align}
Similarly,
\begin{align}\label{3.21}
&\int_{ \hat K^\prime_{\gamma,0}}|g(z_1,z_2)||z_2|^{-1}\mu dV(z_1,z_2)\nonumber\\\lesssim&\mu_1(\hat K_\gamma)\left(\int_{\mathbb  D}\frac{|\int_{\hat K_\gamma}|g(z_2t,z_2)|\mu_1(t)dV(t)|^{p^\prime}}{(\mu_1(\hat K_\gamma))^{p^\prime}}\mu_2(z_2) |z_2|^2dV(z_2)\right)^{\frac{1}{p^\prime}}\left(\int_{\mathbb D}\mu_2(z_2) |z_2|^{2-p}dV(z_2)\right)^{\frac{1}{p}}.
\end{align}
Set ${f^*}(t,z_2)=f(z_2t,z_2)$ and $g^*(t,z_2)=g(z_2t,z_2)$. Recall the boundedness of $\mathcal M_{\mathcal T_m,\mu_1}$ from the proof of Lemma 2.6 and the fact that $(\nu_1(\hat K_\gamma))^{2-p}\leq (\nu_1( K_\gamma))^{2-p}$ for $p\geq 2$. Applying this facts, substituting (\ref{3.20}) and (\ref{3.21}) into (\ref{3.2}), and following the computation for the operator $Q^{1,1}_{m,n,\nu}$ then yields that for $p\geq 2$,
\begin{align}
&\left\langle Q^{1,0}_{m,n,\nu} f(z_1,z_2), g(z_1,z_2)\mu\right\rangle\nonumber\\\lesssim&\left(\langle|z_2|^{2-p}\mu_2\rangle^{dV}_{\mathbb D}(\langle\nu_2\rangle^{dV}_{\mathbb D})^{p-1}\right)^{\frac{1}{p}}\sup_{0\neq\gamma\in \mathcal T_m}\langle\mu_1\rangle^{dV}_{\hat K_\gamma}\left(\langle\nu_1\rangle^{dV}_{\hat K_\gamma}\right)^{p-1}\times\nonumber\\&\|\mathcal M_{\mathcal T_m,\nu_1}(|f^*(\cdot,z_2)|)\|_{L^p(\mathbb D^2,|z_2|^2\nu_1\nu_2)}\|\mathcal M_{\mathcal T_m,\mu_1}(|g^*(\cdot,z_2)|)\|_{L^{p^\prime}(\mathbb D^2,|z_2|^2\mu_1\mu_2)}
\nonumber\\\lesssim&\left(\langle|z_2|^{2-p}\mu_2\rangle^{dV}_{\mathbb D}(\langle\nu_2\rangle^{dV}_{\mathbb D})^{p-1}\right)^{\frac{1}{p}}\sup_{0\neq\gamma\in \mathcal T_m}\langle\mu_1\rangle^{dV}_{\hat K_\gamma}\left(\langle\nu_1\rangle^{dV}_{\hat K_\gamma}\right)^{p-1}pp^\prime\|f^*\|_{L^p(\mathbb D^2,|z_2|^2\nu_1\nu_2)}\|g^*\|_{L^{p^\prime}(\mathbb D^2,|z_2|^2\mu_1\mu_2)}
\nonumber\\\lesssim&\left(\langle|z_2|^{2-p}\mu_2\rangle^{dV}_{\mathbb D}(\langle\nu_2\rangle^{dV}_{\mathbb D})^{p-1}\right)^{\frac{1}{p}}\sup_{0\neq\gamma\in \mathcal T_m}\langle\mu_1\rangle^{dV}_{\hat K_\gamma}\left(\langle\nu_1\rangle^{dV}_{\hat K_\gamma}\right)^{p-1}pp^\prime\|f\|_{L^p(\mathbb H,\nu)}\|g\|_{L^{p^\prime}(\mathbb H,\mu)}.
\end{align}
The same duality argument as in (\ref{3.12}) implies that for $1<p<2$,

\begin{align*}
&\left\langle Q^{1,0}_{m,n,\nu} f(z_1,z_2), g(z_1,z_2)\mu\right\rangle\nonumber\\\lesssim&\left(\langle|z_2|^{2-p}\mu_2\rangle^{dV}_{\mathbb D}(\langle\nu_2\rangle^{dV}_{\mathbb D})^{p-1}\right)^{\frac{1}{p}}\left(\sup_{0\neq\gamma\in \mathcal T_m}\langle\mu_1\rangle^{dV}_{\hat K_\gamma}\left(\langle\nu_1\rangle^{dV}_{\hat K_\gamma}\right)^{p-1}\right)^{\frac{1}{p-1}}pp^\prime\|f\|_{L^p(\mathbb H,\nu)}\|g\|_{L^{p^\prime}(\mathbb H,\mu)}.
\end{align*}
Combining these inequalities, we obtain
\begin{align}\label{3.251}&\|  Q^{1,0}_{m,n,\nu}\|_{L^p(\mathbb H,\nu dV)\to L^p(\mathbb H,\mu dV)}\nonumber\\\lesssim& pp^\prime\left(\langle|z_2|^{2-p}\mu_2\rangle^{dV}_{\mathbb D}(\langle\nu_2\rangle^{dV}_{\mathbb D})^{p-1}\right)^{\frac{1}{p}}\left(\sup_{0\neq\gamma\in \mathcal T_m}\langle\mu_1\rangle^{dV}_{\hat K_\gamma}\left(\langle\nu_1\rangle^{dV}_{\hat K_\gamma}\right)^{p-1}\right)^{\max\{1,\frac{1}{p-1}\}}=pp^\prime[\mu,\nu]^{1,0}_p\end{align}
By a similar argument, one can obtain the estimate for $\|Q^{0,1}_{m,n,\nu}\|$:
\begin{align}\label{3.26}&\|  Q^{0,1}_{m,n,\nu}\|_{L^p(\mathbb H,\nu dV)\to L^p(\mathbb H,\mu dV)}\nonumber\\\lesssim&pp^\prime\left(\sup_{0\neq \eta\in \mathcal T_n}\langle|z_2|^{2-p}\mu_2\rangle^{dV}_{\hat K_\eta}(\langle\nu_2\rangle^{dV}_{\hat K_\eta})^{p-1}\right)^{\max\{1,\frac{1}{p-1}\}}\left(\langle\mu_1\rangle^{dV}_{\mathbb D}\left(\langle\nu_1\rangle^{dV}_{\mathbb D}\right)^{p-1}\right)^{\frac{1}{p}}\nonumber\\=&pp^\prime[\mu,\nu]^{0,1}_p.\end{align}
Combining (\ref{3.80}), (\ref{3.200}), (\ref{3.251}) and (\ref{3.26}), we obtain the upper bound in Theorem \ref{t:main}:
\[\|P^+\|_{L^p(\mathbb H,\mu)}\lesssim ([\mu,\nu]^{0,0}_p)^{1/p}+pp^\prime([\mu,\nu]^{1,0}_p+[\mu,\nu]^{0,1}_p)+(pp^\prime)^2([\mu,\nu]^{1,1}_p)^{\max\{1,\frac{1}{p-1}\}}\]}
\subsection{Proof for the lower bound}
Now we turn to show the lower bound $$[\mu,\nu]_p^{\frac{1}{2p}}\lesssim \|P\|_{L^p(\mathbb H,\mu dV)}$$ in Theorem \ref{t:main}. By the proof of Lemma \ref{l:2.8}, 
\begin{equation}
\|P\|_{L^p(\mathbb H,\mu dV)}= \|M_{z_2}QM_\nu:L^p(\mathbb H,\nu dV)\to L^p(\mathbb H,\mu dV)\|.
\end{equation}
It suffices to show that $[\mu,\nu]_p\leq \|M_{z_2}QM_\nu:L^p(\mathbb H,\nu dV)\to L^p(\mathbb H,|z_2|^{-p}\mu dV)\|^{2p}$. For simplicity, we set $\mathcal A:=\|M_{z_2}QM_\nu:L^p(\mathbb H,\nu dV)\to L^p(\mathbb H,|z_2|^{-p}\mu dV)\|$. Set $|z_2|^{-p}\mu=\mu_p$. If $\mathcal A<\infty$, then we have a weak-type $(p,p)$ estimate:
\begin{equation}\label{3.19}
\mu_p\{(w_1,w_2)\in\mathbb H:|M_{z_2}QM_\nu f(w_1,w_2)|>\lambda\}\lesssim\frac{\mathcal A^{p}}{\lambda^p}\|f\|^p_{L^p(\mathbb H,\nu dV)}.
\end{equation}
We choose $f(w_1,w_2)=1_{\hat K^\prime_{\gamma,\eta}}(w_1,w_2)$ with $\gamma$ and $\eta$ to be determined. Then
\begin{align}
&|M_{z_2}QM_\nu1_{\hat K^\prime_{\gamma,\eta}}(z_1,z_2)|\nonumber\\=&\left|\int_{ \hat K^\prime_{\gamma,\eta}}\frac{1}{\pi^2(1-\frac{z_1\bar{w}_1}{z_2\bar w_2})^2(1-z_2\bar w_2)^2}\nu(w_1,w_2) dV(w_1,w_2)\right|
\nonumber\\=&\left|\int_{ \hat K_{\gamma,\eta}}\frac{1}{\pi^2(1-\frac{z_1}{z_2}\bar t_1)^2(1-z_2\bar w_2)^2}\nu(t_1w_2,w_2) |w_2|^2 dV(t_1,w_2)\right|
\nonumber\\=&\left|P_{\mathbb D^2}(|w_2|^{2}{\nu(t_1w_2,w_2)}1_{\hat K_\gamma\times\hat K_\eta}(t_1,w_2))(z_1/z_2,z_2)\right|.
\end{align} 
Here $P_{\mathbb D^2}$ is the Bergman projection on the polydisc $\mathbb D^2$.

Recall that for a point $z\in \mathbb D$ and a tree structure $\mathcal T$, the generation $\text{gen}(z)$ equals $N$ if $z\in K^N_j$ for some $j$.
By \cite[Lemma 5]{Bekolle}, there exists an integer $N$ so that for $\gamma\in \mathcal{T}_m$ with $\text{gen}(\gamma)>N$, there is a $\gamma^\prime\in \mathcal T_{m^\prime}$ with $\text{gen}(\gamma)=\text{gen}(\gamma^\prime)$ such that for any fixed $z\in \hat K_{\gamma^\prime}$ and all $w\in \hat K_{\gamma}$ there holds,
$$(1-z\bar w)^{-2}=(1-z\bar \gamma)^{-2}+((1-z\bar w)^{-2}-(1-z\bar \gamma)^{-2}),$$
where $|(1-z\bar w)^{-2}-(1-z\bar \gamma)^{-2}|\leq 2^{-1}|1-z\bar \gamma|^{-2}$ and $|1-z\bar \gamma|^{2}\approx |\hat K_\gamma|$. Moreover, an elementary geometric argument yields that $\arg((1-z\bar w)^{-2},(1-z\bar \gamma)^{-2})\leq \pi/6$ for all $w\in \hat K_{\gamma}$.
Thus for $(\gamma,\eta)\in \mathcal T_m\times \mathcal T_n$ with $\text{gen}(\gamma), \text{gen}(\eta)>N$, there is a $(\gamma^\prime,\eta^\prime)\in \mathcal T_{m^\prime}\times\mathcal T_{n^\prime}$ with $\text{gen}(\gamma)=\text{gen}(\gamma^\prime)$ and $\text{gen}(\eta)=\text{gen}(\eta^\prime)$ such that for any fixed $(z_1/z_2,z_2)\in \hat K_{\gamma^\prime}\times \hat K_{\eta^\prime}$ there holds:
$$\arg\left(\left(1-\frac{z_1}{z_2}\bar t_1\right)^{-2}(1-{z_2}\bar w_2)^{-2},\left(1-\frac{z_1}{z_2}\bar \gamma\right)^{-2}(1-{z_2}\bar \eta)^{-2}\right)\leq \pi/3,$$
for all $(t_1,w_2)\in \hat K_\gamma\times\hat K_\eta$. Hence 
\begin{align*}
&|P_{\mathbb D^2}(|w_2|^{2}{\nu(t_1w_2,w_2)}1_{\hat K_\gamma\times\hat K_\eta}(t_1,w_2))(z_1/z_2,z_2)|\\=&\left|\int_{ \hat K_{\gamma,\eta}}\frac{1}{\pi^2(1-\frac{z_1}{z_2}\bar t_1)^2(1-z_2\bar w_2)^2}\nu(t_1w_2,w_2) |w_2|^2 dV(t_1,w_2)\right|\\\geq&16^{-1}\int_{ \hat K_{\gamma,\eta}}\frac{1}{\pi^2|1-\frac{z_1}{z_2}\bar \gamma|^2|1-z_2\bar \eta|^2}\nu(t_1w_2,w_2) |w_2|^2 dV(t_1,w_2) \\>&c_1 \langle |w_2|^2\nu(t_1w_2,w_2)\rangle^{dV}_{\hat K_\gamma\times \hat K_\eta},
\end{align*}for some constant $c_1$. 
Thus via the biholomorphism between $\mathbb D\times \mathbb D^*$ and $\mathbb H$, the following containment holds:
\begin{equation}
\hat K^{\prime}_{\gamma^\prime, \eta^\prime}\subseteq \{(w_1,w_2)\in\mathbb H:|M_{z_2}QM_\nu f(w_1,w_2)|>c_1\langle |w_2|^2\nu(t_1w_2,w_2)\rangle^{dV}_{\hat K_\gamma\times \hat K_\eta}\}.
\end{equation}
{By \cite[Lemma 4]{Bekolle}, there holds that $\nu(\mathbb H) <\infty$. Hence $$\langle |w_2|^2\nu(t_1w_2,w_2)\rangle^{dV}_{\hat K_\gamma\times \hat K_\eta}=\langle|w_2|^2\nu\rangle_{\hat K^\prime_{\mathbb \gamma,\eta}}^{du}<\infty.$$} Inequality (\ref{3.19}) then implies 
\begin{equation}
\mu_p(\hat K^{\prime}_{\gamma^\prime, \eta^\prime})\leq\mathcal A^p\left(\langle|w_2|^2\nu\rangle_{\hat K^\prime_{\mathbb \gamma,\eta}}^{du}\right)^{-p}\nu(\hat K^\prime_{\gamma,\eta}),
\end{equation}
which is equivalent to $\langle |w_2|^{2-p}\mu\rangle_{\hat K^{\prime}_{\gamma^\prime, \eta^\prime}}^{du}\left(\langle|w_2|^2\nu\rangle_{\hat K^{\prime}_{\gamma, \eta}}^{du}\right)^{p-1}\lesssim \mathcal A^p$. Since one can interchange the roles of $\gamma, \eta$ and $\gamma^\prime, \eta^\prime$ in the proof of \cite[Lemma 5]{Bekolle}, there holds $$\langle |w_2|^{2-p}\mu\rangle_{\hat K^{\prime}_{\gamma, \eta}}^{du}\left(\langle|w_2|^2\nu\rangle_{\hat K^{\prime}_{\gamma^\prime, \eta^\prime}}^{du}\right)^{p-1}\lesssim \mathcal A^p.$$  Combining these two inequalities, we have
\begin{equation}\label{3.25}
\left(\langle |w_2|^{2-p}\mu\rangle_{\hat K^{\prime}_{\gamma, \eta}}^{du}\left(\langle|w_2|^2\nu\rangle_{\hat K^{\prime}_{\gamma, \eta}}^{du}\right)^{p-1}\right)\left(\langle |w_2|^{2-p}\mu\rangle_{\hat K^{\prime}_{\gamma^\prime, \eta^\prime}}^{du}\left(\langle|w_2|^2\nu\rangle_{\hat K^{\prime}_{\gamma^\prime, \eta^\prime}}^{du}\right)^{p-1}\right)\lesssim \mathcal A^{2p}.
\end{equation}
By H\"older's inequality,
\begin{equation}
u(\hat K^\prime_{\gamma, \eta})^{p}\leq \int_{ \hat K^\prime_{\gamma,\eta}}|w_2|^{2-p}\mu du\left(\int_{ \hat K^\prime_{\gamma,\eta}}|w_2|^2\nu du\right)^{p-1}
\end{equation}
for any $(\gamma,\eta)\in \mathcal T_{m,n}$. Therefore $\langle |w_2|^{2-p}\mu\rangle_{\hat K^{\prime}_{\gamma, \eta}}^{du}\left(\langle|w_2|^2\nu\rangle_{\hat K^{\prime}_{\gamma, \eta}}^{du}\right)^{p-1}\gtrsim 1$ for all $\gamma,\eta\in \mathcal T_{m,n}$. Applying this to (\ref{3.25}) and taking the supremum of the left side of (\ref{3.25}) for $\text{gen}(\gamma)>N$ and $\text{gen}(\eta)>N$, there holds
\begin{equation}\label{3.27}
\sup_{\substack{(\gamma,\eta)\in \mathcal T_{m,n},\\\text{gen}(\gamma),\text{gen}(\eta)>N}}\langle |w_2|^{2-p}\mu\rangle_{\hat K^{\prime}_{\gamma, \eta}}^{du}\left(\langle|w_2|^2\nu\rangle_{\hat K^{\prime}_{\gamma, \eta}}^{du}\right)^{p-1}\lesssim \mathcal A^{2p}.
\end{equation}

We turn to show that (\ref{3.27}) also holds when the supremum is taken over tents where either $\text{gen}(\gamma)\leq N$ or $\text{gen}(\eta)\leq N$.

Suppose that both $\text{gen}(\gamma)\leq N$ and $\text{gen}(\eta)\leq N$. Then $\hat K_\gamma$ and $\hat K_\eta$ are big tents on the unit disk $\mathbb D$ and $|\hat K_\gamma|=|\hat K_\eta|\approx 1$. Set $B_{1/4}=\{z\in \mathbb C:|z|<1/4\}$. Then for any given $z\in \mathbb D$, $|z\bar w|<1/4$ for $w\in B_{1/4}$. Therefore $\text{Arg}((1-z\bar w)^2)\subseteq[-\frac{\pi}{6},\frac{\pi}{6}]$.  Applying this fact, we obtain
{\begin{align}&\left|P_{\mathbb D^2}(|w_2|^2{\nu(t_1w_2,w_2)}1_{B_{1/4}\times B_{1/4}}(t_1,w_2))\left(\frac{z_1}{z_2},z_2\right)\right|\nonumber\\=&\left|\int_{ B_{1/4}\times B_{1/4}}\frac{|w_2|^2}{\pi^2(1-\frac{z_1}{z_2}\bar t_1)^2(1-z_2\bar w_2)^2}{\nu(t_1w_2,w_2)}dV(t_1,w_2)\right|\nonumber\\\geq&16^{-1}\left|\int_{ B_{1/4}\times B_{1/4}}\pi^{-2}{|w_2|^2}{\nu(t_1w_2,w_2)}dV(t_1,w_2)\right|\geq c_2\langle |w_2|^2\nu(t_1w_2,w_2)\rangle^{dV}_{B_{1/4}\times B_{1/4}}\end{align} for some constant $c_2$.
Therefore,
\begin{equation*}
\mathbb D^2=\left\{(z_1,z_2)\in \mathbb D^2:|P_{\mathbb D}(|w_2|^2{\nu(t_1w_2,w_2)}1_{B_{1/4}\times B_{1/4}})(z_1,z_2)|>{c_2}\langle |z_2|^2\nu(t_1z_2,z_2)\rangle^{dV}_{B_{1/4}\times B_{1/4}}\right\}.
\end{equation*}
Let $B^\prime_{1/4,1/4}$ denote the set $\{(w_1,w_2)\in \mathbb H:(\frac{w_1}{w_2},w_2)\in B_{1/4}\times B_{1/4}\}$. Via the bihomomorphism between $\mathbb D\times \mathbb D^*$ and $\mathbb H$, we obtain 
\begin{equation*}
\mu_p(\mathbb H)= \mu_p\left\{(z_1,z_2)\in\mathbb H:| P_{\mathbb H}(\nu1_{B^\prime_{1/4}})(z_1,z_2)|>{c_2\langle |z_2|^2\nu\rangle^{du}_{B^\prime_{1/4,1/4}}}\right\}\leq\frac{ \mathcal A^p\|1_{B^\prime_{1/4}}\|^p_{L^p(\mathbb H,\nu dV)}}{c_2^p\left(\langle |z_2|^2\nu\rangle^{du}_{B^\prime_{1/4,1/4}}\right)^p}.
\end{equation*}
Thus $$\langle |w_2|^{2-p}\mu\rangle_{\mathbb H}^{du}\left(\langle|w_2|^2\nu\rangle_{B^\prime_{1/4,1/4}}^{du}\right)^{p-1}\lesssim \mathcal A^p.$$
Interchanging the role of variables $z$ and $w$, we also have 
\begin{equation*}
\mu_p(B^\prime_{1/4,1/4})= \mu_p\left\{(w_1,w_2)\in B^\prime_{1/4,1/4}:| P_{\mathbb H}(\nu1_{\mathbb H})(w_1,w_2)|>{c_2\langle |w_2|^2\nu\rangle^{du}_{\mathbb H}}\right\}\leq\frac{ \mathcal A^p\|1\|^p_{L^p(\mathbb H,\nu dV)}}{c_2^p\left(\langle |w_2|^2\nu\rangle^{du}_{\mathbb H}\right)^p}.
\end{equation*}
Thus$$\langle|w_2|^{2-p}\mu\rangle_{B^\prime_{1/4,1/4}}^{du}(\langle |w_2|^{2}\nu\rangle_{\mathbb H}^{du})^{p-1}\lesssim \mathcal A^{p}.$$ By H\"older's inequality  $$\langle|w_2|^{2-p}\mu\rangle_{B^\prime_{1/4,1/4}}^{du}\left(\langle|w_2|^2\nu\rangle_{B^\prime_{1/4,1/4}}^{du}\right)^{p-1}\geq 1.$$ Combining these inequalities yields that
\begin{align}
&\langle |w_2|^{2-p}\mu\rangle_{\mathbb H}^{du}(\langle |w_2|^{2}\nu\rangle_{\mathbb H}^{du})^{p-1}\nonumber\\\lesssim&\langle|w_2|^{2-p}\mu\rangle_{B^\prime_{1/4,1/4}}^{du}(\langle |w_2|^{2}\nu\rangle_{\mathbb H}^{du})^{p-1}\langle |w_2|^{2-p}\mu\rangle_{\mathbb H}^{du}\left(\langle|w_2|^2\nu\rangle_{B^\prime_{1/4,1/4}}^{du}\right)^{p-1}\lesssim \mathcal A^{2p}
\end{align}
Therefore, $|\hat K^{\prime}_{\gamma, \eta}|\approx 1$ implies
\begin{align}\label{3.31}
\langle |w_2|^{2-p}\mu\rangle_{\hat K^{\prime}_{\gamma, \eta}}^{du}\left(\langle|w_2|^2\nu\rangle_{\hat K^{\prime}_{\gamma, \eta}}^{du}\right)^{p-1}\lesssim
\langle |w_2|^{2-p}\mu\rangle_{\mathbb H}^{du}(\langle |w_2|^{2}\nu\rangle_{\mathbb H}^{du})^{p-1}\lesssim \mathcal A^{2p}.
\end{align} }

For the case $\text{gen}(\gamma)\leq N$ and $\text{gen}(\eta)> N$, we combine the arguments for both the big tents and the small tents. There exists an $\eta^\prime$ with $\text{gen}(\eta)=\text{gen}(\eta^\prime)$ such that for all $\frac{z_1}{z_2}\in \mathbb D$ and $z_2\in \hat K_{\eta^\prime}$, there holds:
\begin{align*}
&|P_{\mathbb D^2}(|w_2|^{2}{\nu(t_1w_2,w_2)}1_{ B_{1/4}\times\hat K_\eta}(t_1,w_2))(z_1/z_2,z_2)|\\=&\left|\int_{ B_{1/4}\times\hat K_\eta}\frac{|w_2|^2\nu(t_1w_2,w_2)}{\pi^2(1-\frac{z_1}{z_2}\bar t_1)^2(1-z_2\bar w_2)^2}dV(t_1,w_2)\right|\\\geq&16^{-1}\int_{ B_{1/4}\times\hat K_\eta}\frac{|w_2|^2\nu(t_1w_2,w_2)}{\pi^{2}|1-z_2\bar \eta|^2}dV(t_1,w_2)> c_3\langle |w_2|^2\nu(t_1w_2,w_2)\rangle^{dV}_{\hat K_0\times \hat K_\eta},
\end{align*}
for some constant $c_3$. Set $B^\prime_{1/4,\eta}=\{(w_1,w_2)\in \mathbb H:(\frac{w_1}{w_2},w_2)\in B_{1/4}\times \hat K_\eta\}$.
Via the biholomorphism between $\mathbb D\times \mathbb D^*$ and $\mathbb H$ again, the following containment holds:
\begin{equation}\label{3.33}
\hat K^{\prime}_{0, \eta^\prime}\subseteq \left\{(w_1,w_2)\in\mathbb H:|M_{z_2}QM_\nu 1_{B^\prime_{1/4,\eta}}(w_1,w_2)|>\frac{c_3}{32}\langle |w_2|^2\nu(t_1w_2,w_2)\rangle^{dV}_{\hat K_0\times \hat K_\eta}\right\}.
\end{equation}
Applying the proof for inequalities (\ref{3.27}) and (\ref{3.31}) to (\ref{3.33}) gives
\begin{equation}
\langle |w_2|^{2-p}\mu\rangle_{\hat K^{\prime}_{\gamma, \eta}}^{du}\left(\langle|w_2|^2\nu\rangle_{\hat K^{\prime}_{\gamma, \eta}}^{du}\right)^{p-1}\lesssim \langle |w_2|^{2-p}\mu\rangle_{\hat K^{\prime}_{0, \eta}}^{du}\left(\langle|w_2|^2\nu\rangle_{\hat K^{\prime}_{0, \eta}}^{du}\right)^{p-1}\lesssim \mathcal A^{2p}.
\end{equation}
The last case $\text{gen}(\eta)\leq N_n$ and $\text{gen}(\gamma)>N_m$ follows from a similar argument with the role of $\gamma$ and $\eta$ interchanged. Combining all these estimates, we obtain the desired lower bound:
\begin{equation}
[\mu,\nu]_p=\sup_{(\gamma,\eta)\in \mathcal T_{m,n}}\langle |w_2|^{2-p}\mu\rangle_{\hat K^{\prime}_{\gamma, \eta}}^{du}\left(\langle|w_2|^2\nu\rangle_{\hat K^{\prime}_{\gamma, \eta}}^{du}\right)^{p-1}\lesssim \mathcal A^{2p},
\end{equation}
which completes the proof of Theorem \ref{t:main}.
\section{Examples}
We begin by providing a sharp example for the upper bound estimate in Theorem \ref{t:main}.
\subsection{A sharp example for the upper bound} We give an example for the case $1<p\leq 2$ here. The case $p>2$ follows from a duality argument. The idea is based on the construction of the sharp examples in \cite{Pott} and \cite{Rahm}. {Recall $B_{1/4}=\{z\in \mathbb C:|z|<1/4\}$. Given a number $1>s>0$, we set \begin{align}\mu(w_1,w_2)&=|w_2|^{p-2}\frac{|(1-w_1/w_2)(1-w_2)|^{2(p-1)(1-s)}}{|w_1/w_2|^{2-s}|w_2|^{2-s}},\\
 f(w_1,w_2)&=\bar w_2\mu^{\frac{1}{1-p}}(w_1,w_2)1_{T_{\frac{1}{2}}}({w_1}/{w_2})1_{T_{\frac{1}{2}}}(w_2).\end{align}
 Then for $(t_1,t_2)\in \mathbb D^2$ with $T^\prime_{t_1,t_2}$ not intersecting the tent $T^\prime_{|t_1|,|t_2|}$ and away from the point $(0,0)$, there holds $\mu(w_1,w_2)\approx 1$ and hence $\langle\mu|w_2|^{2-p} \rangle^{du}_{T^\prime_{t_1,t_2}}\left(\langle |w_2|^{2}\nu\rangle^{du}_{T^\prime_{t_1,t_2}}\right)^{p-1}\approx 1$.
 
 When $T^\prime_{t_1,t_2}$ intersects the tent $T^\prime_{|t_1|,|t_2|}$ with $|t_1|, |t_2|\geq 1/2$, there exists a positive constant $c>0$ such that $T^\prime_{t_1,t_2}\subseteq T^\prime_{c|t_1|,c|t_2|}$ and $u(T^\prime_{t_1,t_2})\approx u(T^\prime_{c|t_1|,c|t_2|})$. Thus 
 \begin{align}\label{4.3}
& \int_{T^\prime_{t_1,t_2}} |w_2|^{2-p}\mu(w_1,w_2) du(w_1,w_2)\nonumber\\\lesssim &\int_{T^\prime_{c|t_1|,c|t_2|}} |w_2|^{2-p}\mu(w_1,w_2) du(w_1,w_2)\nonumber\\=&\int_{T_{c|t_1|}\times T_{c|t_2|}}\left|{(1-w_1)(1-w_2)}\right|^{(p-1)(2-2s)}dV(w_1,w_2)\nonumber\\=&\prod_{j=1}^{2}\int_{\{w_j\in \mathbb D:|1-w_j|<1-c|t_j|\}}\left|{1-w_j}\right|^{(p-1)(2-2s)}dV(w_j).
 \end{align}}
 Using the changes of variables $z_j=i\frac{1-{w_j}}{1+w_j}$, we have
 \begin{align}
&\int_{\{w_j\in \mathbb D:|1-w_j|<1-{c}|t_j|\}}\left|\frac{1-w_j}{1+w_j}\right|^{(p-1)(2-2s)}dV(w_j)\nonumber\\\approx&
\int_{\{z_j\in \mathbb C:|z_j|<1-{c}|t_j|,\text{Im}z_j>0\}}\left|z_j\right|^{(p-1)(2-2s)}dV(z_j)\approx \frac{(1-{c}|t_j|)^{(p-1)(2-2s)+2}}{(p-1)(2-2s)+2}.
 \end{align}
Thus
 \begin{equation}
\int_{T^\prime_{t_1,t_2}} |w_2|^{2-p}\mu(w_1,w_2) du(w_1,w_2)\approx\prod_{j=1}^{2}\frac{(1-{c}|t_j|)^{(p-1)(2-2s)+2}}{(p-1)(2-2s)+2}.
\end{equation}
Similarly, for $\nu=|w_2|^{-p^\prime}\mu^{\frac{-p^\prime}{p}}$,
 \begin{equation}
\int_{T^\prime_{t_1,t_2}} |w_2|^2\nu(w_1,w_2) du(w_1,w_2)\approx\prod_{j=1}^{2}\frac{(1-{c}|t_j|)^{2s}}{2s}.
\end{equation}
Since $1<p\leq 2$ and $0<s<1$, we have $${[\mu,\nu]^{1,1}_p}=\sup_{\substack{t_1,t_2\in \mathbb D\\|t_1|,|t_2|\geq 1/2}}\langle|w_2|^{2-p}\mu\rangle^{du}_{T^\prime_{t_1,t_2}}\left(\langle|w_2|^2\nu\rangle_{T^{\prime}_{t_1, t_2}}^{du}\right)^{p-1}\lesssim s^{-2(p-1)}.$$
{Moreover,
\begin{align}
&	[\mu,\nu]^{0,0}_p=\langle |w_2|^{2-p}\mu\rangle^{du}_{\mathbb H}\left(\langle\nu|w_2|^2 \rangle^{du}_{\mathbb H}\right)^{p-1}\approx s^{-2p};
\nonumber\\&	[\mu,\nu]^{1,0}_p=\left(\langle|z_2|^{2-p}\mu_2\rangle^{dV}_{\mathbb D}(\langle\nu_2\rangle^{dV}_{\mathbb D})^{p-1}\right)^{\frac{1}{p}}\left(\sup_{\substack{z\in \mathbb D,\\|z|>{1}/{2}}}\langle\mu_1\rangle^{dV}_{T_z}\left(\langle\nu_1\rangle^{dV}_{T_z}\right)^{p-1}\right)^{\max\{1,\frac{1}{p-1}\}}\approx s^{-2};\nonumber\\&	[\mu,\nu]^{0,1}_p=\left(\sup_{\substack{z\in \mathbb D,\\|z|>{1}/{2}}}\langle|z_2|^{2-p}\mu_2\rangle^{dV}_{T_z}(\langle\nu_2\rangle^{dV}_{T_z})^{p-1}\right)^{\max\{1,\frac{1}{p-1}\}}\left(\langle\mu_1\rangle^{dV}_{\mathbb D}\left(\langle\nu_1\rangle^{dV}_{\mathbb D}\right)^{p-1}\right)^{\frac{1}{p}}\approx s^{-2}.\nonumber\end{align}
Thus the upper bound
$$([\mu,\nu]^{0,0}_p)^{\frac{1}{p}}+pp^\prime([\mu,\nu]^{1,0}_p+[\mu,\nu]^{0,1}_p)+(pp^\prime)^2([\mu,\nu]^{1,1}_p)^{\max\{1,\frac{1}{p-1}\}}\lesssim s^{-2}.$$}
When $w_1/w_2,w_2\in T_{1/2}$,  there hold {$|w_1|/|w_2|,|w_2|\approx 1$.} Therefore
\begin{align}
\|f\|^p_{L^p(\mathbb H,\mu)}\approx\int_{T^\prime_{1/2,1/2}}\mu^{\frac{p}{1-p}}(w_1,w_2)\mu(w_1,w_2)dV(w_1,w_2)\approx\langle|w_2|^2\nu\rangle_{T^{\prime}_{1/2, 1/2}}^{du}\approx s^{-2}.
\end{align}
{For $z_1/z_2, z_2\in B_{1/4}$, we claim that \begin{equation}\label{4.8}
|P(f)(z_1,z_2)|\gtrsim|z_2|^{-1}\langle f\rangle^{dV}_{T^\prime_{1/2,1/2}}.\end{equation}
Note that for $w\in T_{\frac{1}{2}}$ and $z\in B_{1/4}$, there holds that $|1-z\bar w|\approx 1$ and $$\arg\{(1-z\bar w),1\}\in (-\arcsin (1/4),\arcsin (1/4)).$$}
Using these facts and the formula (\ref{K}) for the Bergman projection, we have
{\begin{align}
|P(f)(z_1,z_2)|&=\left|\int_{T^\prime_{1/2,1/2}}\frac{f(w_1,w_2)}{\pi^2z_2\bar w_2(1-\frac{z_1\bar w_1}{z_2 \bar w_2})^2(1-z_2\bar w_2)^2}dV(w_1,w_2)\right|\nonumber\\
&=|z_2|^{-1}\left|\int_{T_{1/2,1/2}}\frac{f(w_1w_2,w_2)w_2}{\pi^2(1-\frac{z_1}{z_2}\bar w_1)^2(1-z_2\bar w_2)^2}dV(w_1,w_2)\right|\nonumber\\
&\gtrsim|z_2|^{-1}\int_{T_{1/2,1/2}}|1-w_1/w_2|^{2s-2}|1-w_2|^{2s-2}dV(w_1,w_2)
\nonumber\\
&\approx|z_2|^{-1}\langle f\rangle^{dV}_{T^\prime_{1/2,1/2}}.
\end{align}}
{Since $|w_1|/|w_2|,|w_2|\approx 1$ for $(w_1,w_2)\in T^\prime_{1/2,1/2}$, we have
\begin{align}
\langle f\rangle_{T^\prime_{1/2,1/2}}&\approx \int_{T^\prime_{1/2,1/2}} \left|{(1-w_1/w_2)(1-w_2)}\right|^{2(s-1)}dV(w_1,w_2)\approx\langle|w_2|^2\nu\rangle_{T^{\prime}_{1/2, 1/2}}^{du}\approx s^{-2}.
\end{align}}
{Thus $|P(f)(z_1,z_2)|\gtrsim |z_2|^{-1}s^{-2}$.
Moreover,
\begin{align}
\|Pf\|^p_{L^p(\mathbb H,\mu dV)}&=\int_{\mathbb H}|Pf(z_1,z_2)|^p\mu(z_1,z_2) dV(z_1,z_2)\nonumber\\
&\geq s^{-2p}\int_{\{(z_1/z_2,z_2)\in B_{1/4}\times B_{1/4}\}}|z_2|^{-p}\mu(z_1,z_2) dV(z_1,z_2)
\nonumber\\
&\approx s^{-2p}\int_{\{(t,z_2)\in B_{1/4}\times B_{1/4}\}}|t|^{s-2}|z_2|^{s-2} dV(t,z^\prime_2)\approx s^{-2p-2}.
\end{align}
Thus the desired estimates holds:\[\frac{\|Pf\|^p_{L^p(\mathbb H,\mu dV)}}{\|f\|^p_{L^p(\mathbb H,\mu dV)}}\gtrsim s^{-2p}\gtrsim\left(([\mu,\nu]^{0,0}_p)^{\frac{1}{p}}+pp^\prime([\mu,\nu]^{1,0}_p+[\mu,\nu]^{0,1}_p)+(pp^\prime)^2([\mu,\nu]^{1,1}_p)^{\max\{1,\frac{1}{p-1}\}}\right)^p.\]}
\subsection{$L^p$ regularity of the Bergman projection on the Hartogs triangle} If weight $\mu$ is identically 1, then $\mu dV$ is the Lebesgue measure on the Hartogs triangle, and $\|P^+\|_{L^p(\mathbb H,\mu)}$ is the unweighted $L^p$ norm of the Bergman projection. Chakrabarti and Zeytuncu showed in \cite{DebrajY} that the Bergman projection on the Hartogs triangle is $L^p$ regular if and only if $\frac{4}{3}<p<4$. Using Theorem \ref{t:main}, we give an alternative proof of this $L^p$ regularity result.

Set $\mu\equiv1$. Then $\nu=|w_2|^{-p^\prime}$ and $$[\mu,\nu]_p=\sup_{\substack{(\gamma,\eta)\in \mathcal T_{m,n}\\1\leq m,n\leq M}}\langle |w_2|^{2-p}\rangle_{\hat K^{\prime}_{\gamma, \eta}}^{du}\left(\langle|w_2|^{2-p^\prime}\rangle_{\hat K^{\prime}_{\gamma, \eta}}^{du}\right)^{p-1}.$$
When $p\geq{4}$ or $p\leq 4/3$, we have
$\int_{\mathbb H}|w_2|^{2-p}du\int_{\mathbb H}|w_2|^{2-p^\prime}du=\infty.$
Thus $[\mu,\nu]_p=\infty$ for $p\notin(\frac{4}{3},4)$. By Theorem \ref{t:main}, the Bergman projection $P$ is not bounded on $L^p(\mathbb H)$. 

When $p\in (\frac{4}{3},4)$, we have {for $(\gamma,\eta)\in \mathcal T_{m,n}$ where $1\leq m,n\leq M$ that 
\begin{align*}
\langle |w_2|^{2-p}\rangle_{\hat K^{\prime}_{\gamma, \eta}}^{du}\left(\langle|w_2|^{2-p^\prime}\rangle_{\hat K^{\prime}_{\gamma, \eta}}^{du}\right)^{p-1}
&=\langle |w_2|^{2-p}\rangle_{\hat K_{ \eta}}^{dV}\left(\langle|w_2|^{2-p^\prime}\rangle_{\hat K_{\eta}}^{dV}\right)^{p-1}.
\end{align*}
Therefore for all $w_2\in \hat K_\eta$ with {$|\eta|\geq 1/2$}, we have  $|w_2|\gtrsim 1$  and 
\begin{equation}
\langle |w_2|^{2-p}\rangle_{\hat K_{ \eta}}^{dV}\left(\langle|w_2|^{2-p^\prime}\rangle_{\hat K_{\eta}}^{dV}\right)^{p-1}\approx 1.
\end{equation}
When $\eta=0$, $\hat K_{\eta}=\mathbb D$ and 
\begin{equation}
\langle |w_2|^{2-p}\rangle_{\mathbb D}^{dV}\left(\langle|w_2|^{2-p^\prime}\rangle_{\mathbb D}^{dV}\right)^{p-1}=[\mu,\nu]^{0,0}_p=\frac{2}{4-p}\left(\frac{2(p-1)}{3p-4}\right)^{p-1}.
\end{equation}
Similarly, we have the estimates for the other $[\mu,\nu]^{i,j}_p$:
\begin{align}
&	[\mu,\nu]^{1,0}_p\approx \left(\langle|w_2|^{2-p}\rangle^{dV}_{\mathbb D}(\langle|w_2|^{2-p^\prime}\rangle^{dV}_{\mathbb D})^{p-1}\right)^{\frac{1}{p}}= \left(\frac{2}{4-p}\left(\frac{2(p-1)}{3p-4}\right)^{p-1}\right)^{1/p};\nonumber\\&	[\mu,\nu]^{0,1}_p\approx 	[\mu,\nu]^{1,1}_p\approx 1.\nonumber\end{align}
Hence for $p\in (3/4,4)$, \begin{align}&([\mu,\nu]^{0,0}_p)^{\frac{1}{p}}+pp^\prime([\mu,\nu]^{1,0}_p+[\mu,\nu]^{0,1}_p)+(pp^\prime)^2([\mu,\nu]^{1,1}_p)^{\max\{1,\frac{1}{p-1}\}}\nonumber\\\approx& \left(\frac{2}{4-p}\left(\frac{2(p-1)}{3p-4}\right)^{p-1}\right)^{1/p}\approx (4-p)^{-1/p}(3p-4)^{-1/p^\prime}<\infty.\end{align} Theorem \ref{t:main} gives the norm estimate of the Bergman projection $P$ for $\frac{4}{3}<p<4$, 
\begin{align}\label{4.15}\|P\|_{L^p(\mathbb H)}\lesssim (4-p)^{-1/p}(3p-4)^{-1/p^\prime},\end{align}
which implies the blowing up of $\|P\|_{L^p(\mathbb H)}$ as $p\to \frac{4}{3}^+$ or $p\to 4^-$. This fact can also be checked by computing the quotient ${\|P(|z_2|^{-p^\prime}\bar z_2)\|^p_{L^p(\mathbb H)}}/{\||z_2|^{-p^\prime}\bar z_2\|^p_{L^p(\mathbb H)}}$ for the cases $p\to 4^-$ and $p\to \frac{4}{3}^+$. Moreover, the estimate (\ref{4.15}) is sharp in the sense that 
\[\frac{\|P(|z_2|^{-p^\prime}\bar z_2)\|_{L^p(\mathbb H)}}{{\||z_2|^{-p^\prime}\bar z_2\|_{L^p(\mathbb H)}}}\approx \frac{1}{(4-p)^{1/p}(3p-4)^{1/p^\prime}}.\]}
\subsection{The case \boldmath${\mu(w_1,w_2)=|w_1|^a|w_2|^b}$} When the weight ${\mu(w_1,w_2)=|w_1|^a|w_2|^b}$, the weight $\nu=|w_1|^{-ap^\prime/p}|w_2|^{-p^\prime(1+b/p)}$. {The singularity of the weight only occurs at places where $w_1$ or $w_2$ vanishes. Hence the blowing up of weights  at both $z_1=0$ and $z_2=0$ can only be captured by computing the average of the weights over the entire Hartogs triangle. Thus $[\mu,\nu]^{0,0}_p$ will be the largest term among $[\mu,\nu]^{i,j}_p$.} By a change of variables we obtain
{\begin{align}\label{4.18}
[\mu,\nu]^{0,0}_p&=\langle |w_1|^a|w_2|^{2-p+b}\rangle_{\mathbb H}^{du}\left(\langle|w_1|^{-ap^\prime/p}|w_2|^{2-p^\prime(1+b/p)}\rangle_{\mathbb H}^{du}\right)^{p-1}\nonumber
\\&=\langle |w_1|^a\rangle_{\mathbb D}^{dV}\left(\langle|w_1|^{-\frac{ap^\prime}{p}}\rangle_{\mathbb D}^{dV}\right)^{p-1}\langle|w_2|^{2+a-p+b}\rangle_{\mathbb D}^{dV}\left(\langle|w_2|^{2-p^\prime(1+\frac{a+b}{p})}\rangle_{\mathbb D}^{dV}\right)^{p-1}.
\end{align}
A computation using polar coordinates implies the following estimate for $[\mu,\nu]^{0,0}_p$:}
\begin{itemize}
\item $
{[\mu,\nu]^{0,0}_p}\approx(a+2)^{-1}(2-\frac{ap^\prime}{p})^{1-p} (4-p+a+b)^{-1}(4-p^\prime(1+\frac{a+b}{p}))^{1-p},
$
for $-2<a<2(p-1)$ and $p-4<a+b<3p-4$;\item ${[\mu,\nu]^{0,0}_p}=\infty$ otherwise.\end{itemize} Rearranging these inequalities, we conclude that the Bergman projection $P$ is $L^p$ regular for $p\geq 2$ if and only if $\max\{1,\frac{a+1}{2},\frac{a+b+4}{3}\}<p<a+b+4$. The $p$ range we obain here is not of form $(\alpha,\frac{\alpha}{\alpha-1})$. This is because, for the case $a$ or $b$ is not zero, the Bergman projection operator is not self-adjoint on $L^p(\mathbb H,\mu dV)$, and is not necessarily $L^2$ bounded. 

\subsection{$L^p$ regularity of the Bergman projection on the generalized Hartogs triangle} In \cite{EM2}, Edholm and McNeal studied the $L^p$ boundedness of the Bergman projection on the generalized Hartogs triangle
$$\mathbb H_{m/n}=\{(z_1,z_2)\in \mathbb C^2:|z_1|^m<|z_2|^n<1\},$$
where $m,n\in\mathbb Z^+$ with $\gcd(m,n)=1$. A crucial step in their paper (see \cite[Proposition 3.4]{EM2}) is to analyze the $L^p$ regularity of the integral operator $\mathcal K_A$
defined by $$\mathcal K_A(f)(z_1,z_2):=\int_{\mathbb H_{m/n}} \frac{|z_2\bar w_2|^A}{|1-z_2\bar w_2|^2|z_2^n\bar w_2^n-z_1^m\bar w_1^m|^2}f(w_1,w_2)dV(w_1,w_2).$$
 Using the proper map $h:(w_1,w_2) \mapsto (w_1^mw_2^{n-1},w_2)$ from $\mathbb H_{m/n}$ to $\mathbb H$, we can relate the $L^p$ norm of $\mathcal K_A$ on $\mathbb H_{m/n}$ to the weighted $L^p$ norm of the absolute Bergman projection on $\mathbb H$:
\begin{align*}
\|\mathcal K_A\|_{L^p(\mathbb H_{m/n})}&=m^{-1}\|M_hP^+M_h:L^p(\mathbb H, \omega_1 dV)\to L^p(\mathbb H, \omega_2 dV)\|\\&=m^{-1}\|P^+M_h:L^p(\mathbb H, \omega_1 dV)\to L^p(\mathbb H, \omega_2h^p dV)\|,
\end{align*}
where the weights $\omega_1(w_1,w_2)=|w_1|^{\frac{2m-2}{m}(p-1)}|w_2|^{\frac{2}{m}(n-1)(1-p)}$, $\omega_2(w_1,w_2)=|w_1|^{-2+\frac{2}{m}}|w_2|^{\frac{2}{m}(n-1)}$, and $h(w_1,w_2)=|w_2|^{A-2n+1}$. Setting $\mu:=\omega_2h^p$ and $\nu:=\omega_1^{-\frac{p^\prime}{p}}|w_2|^{(A-2n)p^\prime}$, we obtain 
\begin{equation}
\|P^+M_h:L^p(\mathbb H, \omega_1 dV)\to L^p(\mathbb H, \omega_2h^p dV)\|=\|Q^+M_\nu:L^p(\mathbb H, \nu dV)\to L^p(\mathbb H, \mu dV)\|.
\end{equation}

Here $\nu$ is no longer equal to $|w_2|^{-p^\prime}\mu^{-\frac{p^\prime}{p}}$ which is the dual weight of $|w_2|^{-p}\mu$ with respect to the measure $u$. Still, $\mu$ and $\nu$ are blowing or vanishing at points where $z_1$ or $z_2$ vanishes. Thus $[\mu,\nu]^{1,1}_p\approx 1$. {Moreover, the constants $[\mu,\nu]^{1,0}_p$ and $[\mu,\nu]^{0,1}_p$ only measure the singularity of the weight in either $z_1$ or $z_2$ variable. Therefore
\begin{align*}
([\mu,\nu]^{0,0}_p)^{1/p}=\left(\langle\mu|w_2|^{2-p} \rangle^{du}_{\mathbb H}\left(\langle |w_2|^{2}\nu\rangle^{du}_{\mathbb H}\right)^{p-1}\right)^{1/p}
\end{align*}
is the main term in the upper bound of the weighted norm of $Q^+M_\nu$.  
By a simple integration, we obtain for $p\in \left(\frac{2n+2m}{Am+2n+2m-2nm},\frac{2n+2m}{2nm-Am}\right)$,
\begin{align}\label{4.22}
[\mu,\nu,]^{0,0}_{p}\approx& (2n-A)^{-p}\left(\frac{2m+2n}{2mn-Am}-p\right)^{-1}\left(\frac{2m+2n}{2mn-Am}-p^\prime\right)^{1-p}\nonumber\\=&(2n-A)^{-p}\left(\frac{2m+2n+Am-2mn}{(p-1)(2mn-Am)}\right)^{1-p}\times\nonumber\\&\left(\frac{2m+2n}{2mn-Am}-p\right)^{-1}\left(p-\frac{2m+2n}{Am+2n+2m-2mn}\right)^{1-p}<\infty.
\end{align}}
Hence, we recover \cite[Proposition 3.4]{EM2}:
\begin{align*}
&\mathcal K_A\text{ is bounded  on } L^p(\mathbb H_{m/n}) \text{ if } p\in \left(\frac{2n+2m}{Am+2n+2m-2nm},\frac{2n+2m}{2nm-Am}\right)\\ &\;\;\;\;\;\;\;\;\;\;\;\;\;\;\text{ whenever } Am+2n+2m-2nm>2nm-Am>0.\end{align*}
and  obtain an $L^p$ norm estimate for such a bounded $\mathcal K_A$:
\begin{align}{\label{4.21}}
\|\mathcal K_A\|_{L^p(\mathbb H_{m/n})}&\lesssim  {m^{-1}([\mu,\nu]^{0,0}_{p})^{\frac{1}{p}}}.
\end{align}
By \cite[Theorem 3.4]{EM2}, the Bergman projection $|P_{\mathbb H_{m/n}}(f)(z)|\lesssim m^2\mathcal K_{A}(|f|)(z)$ with 
$A=2n-1+\frac{1-n}{m}$. Applying (\ref{4.21}) and $(\ref{4.22})$ to this inequality of $P_{\mathbb H_{m/n}}$, we recover the $L^p$ regularity result of the Bergman projection on $\mathbb H_{m/n}$, obtained in \cite[Corollary 4.7]{EM2}, and obtain an estimate for the $L^p$ norm of $P_{\mathbb H_{m/n}}$:
\begin{thm} For $p\in \left(\frac{2m+2n}{m+n+1},\frac{2m+2n}{m+n-1}\right)$,
{\begin{align*}
\|P_{\mathbb H_{m/n}}\|_{L^p(\mathbb H_{m/n})}\lesssim & m^2(p-1)^{\frac{1}{p^\prime}}\left({(2m+2n)}-(m+n-1)p\right)^{-\frac{1}{p}}\times\nonumber\\&\left((n+m+1)p-{(2m+2n)}\right)^{-\frac{1}{p^\prime}}.
\end{align*}}
\end{thm}
{It's not clear for us if the norm estimate above is sharp or not especially when the constant $m$ is large. Each operator $\mathcal K_A$ above corresponds to a sub-Bergman projection induced by an orthogonal decomposition of the Bergman space in \cite{EM2}.  Since the ranges of these sub-Bergman projections are orthogonal to each other, the norm bound obtained using the inequality $|P_{\mathbb H_{m/n}}(f)(z)|\lesssim m^2\mathcal K_{A}(|f|)(z)$ might not be optimal for large $m$.}
\section{Remarks and Generalizations}
\paragraph{1}
The assumption $\mu(z_1,z_2)=\mu_{1}(z_1/z_2)\mu_{2}(z_2)$ in Theorem \ref{t:main} is used only in the proof of Lemma \ref{l:2.7}. 
	Because of this fact, our lower bound in Theorem \ref{t:main} holds without this assumption: 
	\begin{co}
		Let $\mu$ be a weight on $\mathbb H$ and set $\nu=|w_2|^{-p^\prime}\mu^{\frac{-p^\prime}{p}}$. If the Bergman projection $P$ is bounded on the corresponding weighted space $L^p(\mathbb H,\mu dV)$, then $[\mu,\nu]_p<\infty$. Moreover, there holds
		\begin{equation*}
		\|P\|_{L^p(\mathbb H,\mu dV)}\gtrsim 	([\mu,\nu]_p)^{\frac{1}{2p}}.
		\end{equation*}
	\end{co}
Using inequality (\ref{1.11}) and a similar argument for {$Q^{1,1}_{m,n,\nu}$} in the proof of Theorem \ref{t:main}, we can also generalize our upper bound estimate for $P$ and $P^+$ as follows:
\begin{co}
	Let $\mu$ be a weight on $\mathbb H$ and set $\nu=|w_2|^{-p^\prime}\mu^{\frac{-p^\prime}{p}}$. Suppose  the quantities $\|\mathcal M_{\mathcal T^\prime_{m,n},\nu}\|_{L^p(\mathbb H,\nu dV)}$, $\|\mathcal M_{\mathcal T^\prime_{m,n},|w_2|^{-p}\mu}\|_{L^{p^\prime}(\mathbb H,|z_2|^{-p}\mu dV)}$, and $[\mu,\nu]_p$ are all finite. Then the operators $P$ and $P^+$ are bounded on $L^p(\mathbb H,\mu dV)$. Moreover,
	\begin{align*}
	\|P\|_{L^p(\mathbb H,\mu dV)}\leq 	\|P^+\|_{L^p(\mathbb H,\mu dV)}\lesssim& \|\mathcal M_{\mathcal T^\prime_{m,n},\nu}\|_{L^p(\mathbb H,\nu dV)}\|\mathcal M_{\mathcal T^\prime_{m,n},|w_2|^{-p}\mu}\|_{L^{p^\prime}(\mathbb H,|w_2|^{-p}\mu dV)}\\&\times([\mu,\nu]_p)^{\max\{1,\frac{1}{p-1}\}}.
	\end{align*}
\end{co}
In \cite{RFefferman}, Fefferman gave  a sufficient condition for the boundedness of the maximal operator $M_\mu^{(n)}$ on $\mathbb R^n$ defined by 
\begin{equation*}
\mathcal M_\mu^{(n)}(f)(x)=\sup_{x\in R}\frac{\int_R |f(t)|\mu(t) dV(t)}{\int_R\mu(t) dV(t)},
\end{equation*}
where $R$ is any rectangle in $\mathbb R^n$ with sides parallel to the coordinate axes. He showed that, if the weight $\mu$ on $\mathbb R^n$ is uniformly in the class $A_\infty$ in each variable separately, then $\mathcal M_\mu^{(n)}$ is $L^p$ bounded on $L^p(\mathbb R^n,\mu)$ for all $1<p<\infty$. In \cite{Aleman}, Aleman, Pott, and Reguera studied the $B_\infty$ weights on the unit disc which is the analogue of the $A_\infty$ weights in the Bergman setting. Using their results and Fefferman's proof, it is possible to give a sufficient condition for the boundedness of  $\|\mathcal M_{\mathcal T^\prime_{m,n},\nu}\|_{L^p(\mathbb H,\nu dV)}$ and $\|\mathcal M_{\mathcal T^\prime_{m,n},|w_2|^{-p}\mu}\|_{L^p(\mathbb H,|z_2|^{-p}\mu dV)}$  in the corollary above. To obtain an upper bound estimate, one also needs to understand the dependence of the quantities $\|\mathcal M_{\mathcal T^\prime_{m,n},\nu}\|_{L^p(\mathbb H,\nu dV)}$ and $\|\mathcal M_{\mathcal T^\prime_{m,n},|w_2|^{-p}\mu}\|_{L^p(\mathbb H,|z_2|^{-p}\mu dV)}$ on the sufficient condition for the weight $\nu$ and $|z_2|^{-p}\mu$.
\vskip 5pt
\paragraph{2} 
The example in Section 4.1 showed the upper bound estimate in Theorem \ref{t:main} is sharp. It is not clear if the lower bound estimates given in Theorem \ref{t:main}, or in \cite{Pott} and \cite{Rahm} are sharp. It would be interesting to see what a sharp lower bound is in terms of the Bekoll\'e-Bonami constant.
\vskip 5pt
\paragraph{3}  We focus on the weighted estimates for the Bergman projection on the Hartogs triangle for the simplicity of the computation. In \cite{Rahm}, Rahm, Tchoundja, and Wick obtained the weighted estimates for operators $S_{a,b}$ and $S^+_{a,b}$ defined by
\begin{align*}
S_{a,b}f(z)&:=(1-|z|^2)^a\int_{ \mathbb B_n }\frac{f(w)(1-|w|^2)^b}{(1-z\bar w)^{n+1+a+b}}dV(w);
\\S^+_{a,b}f(z)&:=(1-|z|^2)^a\int_{ \mathbb B_n }\frac{f(w)(1-|w|^2)^b}{|1-z\bar w|^{n+1+a+b}}dV(w),
\end{align*}
on the weighted space $L^p(\mathbb B_n,(1-|w|^2)^b\mu dV)$. Using the methods in this paper, it is possible to obtain weighted estimates for analogues of $S_{a,b}$ and $S^+_{a,b}$ in the Hartogs triangle setting. When the domain $\Omega$ is covered by the polydisc through a rational proper holomorphic map as in \cite{Liwei19}, an induced dyadic structure on $\Omega$ can be obtained via the proper map. One direction for generalization is to obtain the Bekoll\'e-Bonami type estimates for the Bergman projection, and analogues of $S_{a,b}$ and $S^+_{a,b}$ on such a domain $\Omega$.
  \vskip 5pt
  \paragraph{4} In the proof of Theorem \ref{t:main}, the positive dyadic operator $Q^+_{m,n,\nu}$ is used to relate the Bergman projection to the maximal operator. The constant $\frac{p^4}{(p-1)^2}$ appeared in Theorem \ref{t:main} dominates  the $L^p$ and $L^{p^\prime}$  norms of the maximal operator on the $\mathbb D^2$. In \cite{Cuckovic17}, \v{C}u\v{c}kovi\'c showed that the $L^p$-norm of the Bergman projection on a smooth bounded strongly pseudoconvex domain is dominated by $\frac{p^2}{p-1}$. This fact suggests the possibility to relate the Bergman projection to the maximal function via a dyadic harmonic analysis argument. It would be interesting to see what is the appropriate dyadic structure and the dyadic operator for the Bergman projection on the strongly pseudoconvex domain, and establish Bekoll\'e-Bonami estimates for weighted $L^p$ norm of the projection. 
    \vskip 5pt
  {\paragraph{5} 
  Although the main result is expressed on the Hartogs triangle, there is a theorem on the bidisc and other product domains in disguise. Note that on the unit ball $\mathbb B_n$, the following estimate \cite{Rahm} holds $$\|P_{\mathbb B_n}\|_{L^p(\mathbb B_n,\mu)}\lesssim \|P^+_{\mathbb B_n}\|_{L^p(\mathbb B_n,\mu)}\lesssim B_p(\mu)^{\max\{1,\frac{1}{p-1}\}}.$$ Fubini's theorem then yields that for a product domain $\Omega=\mathbb B_{n_1}\times\mathbb  B_{n_2}\times\cdots\times \times \mathbb  B_{n_d}$ and a weight $\mu$ on $\Omega$ of the form
  $\mu(\mathbf z_1,\dots,\mathbf z_d)=\mu_1(\mathbf z_1)\cdots \mu_d(\mathbf z_d)$ with $\mathbf z_j\in \mathbb C^{n_j}$ for each $j$, 
  \[\|P_{\Omega}\|_{L^p(\Omega,\mu)}\lesssim\|P^+_{\Omega}\|_{L^p(\Omega,\mu)}\lesssim \prod_{j=1}^{n}  B_p(\mu_j)^{\max\{1,\frac{1}{p-1}\}}.\]
  For more general weights there are challenges in determining the sharp behavior of the weight condition in the multiparameter setting and there are additional issues that appear in these questions.  We plan to undertake a deeper study of these questions in a forthcoming project.  }

\end{document}